\documentclass[a4paper,10pt]{article}
\usepackage[utf8]{inputenc}
\usepackage{amsmath}
\usepackage{amsthm}
\usepackage{amssymb}
\usepackage{color}
\usepackage{textcomp}
\usepackage{authblk}
\usepackage{stmaryrd}
\DeclareMathOperator{\sgn}{sgn}
\newtheorem{lem}{Lemma}[section]

\newtheorem{thm}{Theorem}[section]
\newtheorem{prop}[thm]{Proposition}

\newtheorem{rmk}{Remark}[section]
\theoremstyle{definition}
\newtheorem{defn}{Definition}[section]

\numberwithin{equation}{section}
\date{\empty}

\title{Exponential stability of general 1-D quasilinear systems with source terms for the $C^{1}$ norm under boundary conditions
}
\author{Amaury Hayat}
\affil{Laboratoire Jacques-Louis Lions, UMR CNRS 7598, UPMC Université Paris 6, Sorbonne université.}

\begin{document}

\maketitle

\begin{abstract}
We address the question of the exponential stability for the $C^{1}$ norm of general 1-D quasilinear systems with source terms under boundary conditions. To reach this aim, we introduce the notion of basic $C^{1}$ Lyapunov functions,
 a generic kind of exponentially decreasing function whose existence ensures the exponential stability of the system for the $C^{1}$ norm.
We show that the existence of a basic $C^{1}$ Lyapunov function is subject to two conditions: an interior condition, intrinsic to the system, and a condition on the boundary controls. 
We give explicit sufficient interior and boundary conditions such that the system 
 is exponentially stable for the $C^{1}$ norm and we show that the interior condition is also necessary to the existence of a basic $C^{1}$ Lyapunov function.
Finally, we show that the results conducted in this article are also true under the same conditions for the exponential stability in the $C^{p}$ norm, for any $p\geq1$.
\end{abstract}

\section*{Introduction}

Hyperbolic systems have been studied for several centuries,
 as their importance in representing physical phenomena is 
undeniable.
From gaz dynamics to population evolution through \textcolor{black}{wave} equations and fluid dynamics they are found in many areas.
As they represent the propagation phenomena of numerous physical or industrial systems \cite{Disturbance, GasfanGugat, LeugeringSchmidt}, the issue of their controllability and stability is a major concern, with both theoretical and practical interest. 
If the question of controllability has been well-studied \cite{ControllabilityLi2}, the problem of stabilization under boundary control, however, is only
well known in the particular case of an absence of source term. 
However, in many case neglecting the source term is a crude approximation and reduces greatly the analysis, in particular because it implies that the system can be reduced to decoupled equations or slightly coupled equations (see \cite{deHalleux} for instance). For most physical equations the source term cannot therefore be neglected and the steady-states we aim at stabilizing can be non-uniform with potentially large variations of amplitude (e.g. Saint-Venant equations, see \cite{Chanson} Chapter 5 or \cite{LyapH2}, Euler equations, see \cite{DickGugatLeugering} or \cite{LeugeringH2}, Telegrapher equations, etc.).
Taking into account these nonuniform steady-states and stabilizing them is impossible when not taking the source term into account, although it is an important issue in many applications.
In presence of a source term some results exist for the $H^{2}$ norm (and actually $H^{p},$ $p\geq2$), however, few results exist for the more natural $C^{1}$ norm (and consequent $C^{p}$ norms, $p\geq1$). 
It has to be underlined that for nonlinear systems the stability in these two main topologies are not equivalent as shown in \cite{CoronNguyen}.
In this article we deal with the stability in $C^{1}$ norm of such hyperbolic systems of quasilinear partial differential equations with source term under boundary conditions.
\vspace{\baselineskip}

Several methods are usually used to study the stability of systems. The Lyapunov approach, one of the most famous,
is the one we opted for in this article. 
This approach has the advantage, among others, of guaranteeing some robustness and of being convenient to deal with non-linear problems \cite{CoronNL, Robustness}.
We first introduce the \textit{basic $C^1$ Lyapunov functions}, a kind of natural Lyapunov functions for the $C^1$ norm and we then find a sufficient condition such that the system admits 
a basic $C^1$ Lyapunov function. We show that this sufficient condition is twofold: a first intrinsic condition on the system and a second condition on the boundary controls. 
We show then that this sufficient condition on the system 
is in fact necessary in the general case for the existence of a basic $C^{1}$ Lyapunov function.
\vspace{\baselineskip}\\
The organisation of this paper is as follows: In Section 1, we recall some preliminary properties about 1-D quasilinear hyperbolic system. Section 2 presents an overview of the context and previous results. Section 3 states the main results, which are proven in Section 4. Section 5 presents several remarks and further detail to the results.
\vspace{\baselineskip}\\
\section{\textcolor{black}{Preliminary properties of 1-D quasilinear hyperbolic systems}}
A general quasilinear hyperbolic system can be written as:
\begin{eqnarray}
\mathbf{Y}_{t}+F(\mathbf{Y})\mathbf{Y}_{x}+D(\mathbf{Y})=0, 
\label{system11}\\
\mathcal{B}(\mathbf{Y}(t,0),\mathbf{Y}(t,L))=0,
\label{system12}
\end{eqnarray} 
with $\mathbf{Y}:[0,+\infty)\times[0,L]\to\mathbb{R}^{n}$ and $F:U\to \mathcal{M}_{n}(\mathbb{R})$ and $D:U\to \mathbb{R}^{n}$ where 
$U$ is a non empty connected open set of $\mathbb{R}^{n}$
and $F$ is strictly hyperbolic, i.e. for all $\mathbf{Y}\in U,$ $F(\mathbf{Y})$ has real, distinct eigenvalues. We suppose in addition that these eigenvalues are non-vanishing. $\mathcal{B}$ is a map from $U\times U$ to $\mathbb{R}$ whose form will be precised later on, such that the system (\ref{system11})--(\ref{system12}) is well-posed.
\vspace{\baselineskip}\\
We call $\mathbf{Y}^{*}$ a steady-state of the previous system that we aim at stabilizing. Note that, due to the source term, $\mathbf{Y}^{*}$ is not necessarily uniform and the problem cannot be directly treated as a null stabilization. We therefore
use the following transformation:
\begin{equation}
\mathbf{u}(x,t)=N(x)(\mathbf{Y}(x,t)-\mathbf{Y}^{*}(x)),
\end{equation} 
where $N$ is such that:
\begin{equation}
 NF(\mathbf{Y}^{*})N^{-1}=\Lambda,
 \label{Lambda}
\end{equation} 
where $\Lambda$ is diagonal and corresponds to the eigenvalues of $F(\mathbf{Y}^{*})$. Note that such $N$ exists as the system is strictly hyperbolic. Therefore, the system (\ref{system11})--(\ref{system12}) is equivalent to
\begin{gather}
\mathbf{u}_{t}+A(\mathbf{u},x)\mathbf{u}_{x}+B(\mathbf{u},x)=0,
\label{system21}
\\
\mathcal{B}(N(0)^{-1}\mathbf{u}(0,t)+\mathbf{Y}^{*}(0),N(L)^{-1}\mathbf{u}(L,t)+\mathbf{Y}^{*}(L))=0,
\label{system22}
\end{gather} 
with
\begin{gather}
A(\mathbf{u},x)=N(x)F(\mathbf{Y})N^{-1}(x)=N(x)F(N^{-1}(x)\mathbf{u}+\mathbf{Y}^{*}(x))N^{-1}(x),
\label{defA}\\
B(\mathbf{u},x)=N(F(\mathbf{Y})(\mathbf{Y}^{*}_{x}+(N^{-1})'\mathbf{u})+D(\mathbf{Y})).
\label{defB}
\end{gather}


The difficulty when there is a source term is twofold, and its first aspect can be seen in (\ref{defA}):
 we cannot assume that the steady state $\mathbf{Y}^{*}$ we aim at stabilizing is uniform. Therefore $A$ 
 depends not only on $\mathbf{u}$ but also directly on $x$, and having $A(\mathbf{u}(t,x))$ is different from having $A(\mathbf{u}(t,x),x)$ especially when $\mathbf{u}$ is a perturbation: if $\mathbf{u}$ can still be seen as a perturbation, the dependency on $x$ can no longer be seen itself as a perturbation.\\

Its second aspect is that the source term creates a coupling between the two quantities which is a zero order term that can disturb the Lyapunov function and we will see in Section 2, 3 and 4 that this implies that there does not always exist a simple quadratic Lyapunov function ensuring exponential stability
even when the boundary conditions can be chosen arbitrarly,
 while this phenomenon cannot appear in the absence of source term.\\

From the strict hyperbolicity
we can denote by $m$ the integer such that
\begin{equation}
\text{ }\Lambda_{i}>0\text{ and }\forall i\leq m,\text{ }\Lambda_{i}<0,\text{ }\forall i \in[m+1,n].
\end{equation}
We now denote by $\mathbf{u}_{+}$ the vector of components associated to positive eigenvalues $(u_{1},...,u_{m})^{T}$ and similarly $\mathbf{u}_{-}$ refers to  $(u_{m+1},...,u_{n})^{T}$. In the special cases where $m=0$ or $m=n$ $\mathbf{u}$ is equal to $\mathbf{u}_{-}$ or $\mathbf{u}_{+}$ respectively. 


From now on we will focus on boundary conditions of the form
\begin{equation}
\begin{pmatrix}
\mathbf{u}_{+}(t,0)\\
\mathbf{u}_{-}(t,L)
\end{pmatrix}
= G\begin{pmatrix}
\mathbf{u}_{+}(t,L)\\
\mathbf{u}_{-}(t,0)
\end{pmatrix}.
\label{nonlocal}
\end{equation}
Note that with this boundary conditions the incoming signal is a function of the outgoing signal, which is what is typically expected from a feedback control law and enables the well-posedness of the system (see Theorem \ref{resultatLiRaoWang} later on). 
However the method presented in this article could also be applied to any other boundary conditions of the form (\ref{system12}) that also ensure well-posedness.

We also introduce the consequent first order compatibility conditions for an initial condition $\mathbf{u}^{0}$:
\begin{gather}
\begin{pmatrix}
\mathbf{u}_{+}^{0}(0)\\
\mathbf{u}_{-}^{0}(L)
\end{pmatrix}
= G\begin{pmatrix}
\mathbf{u}_{+}^{0}(L)\\
\mathbf{u}_{-}^{0}(0)
\end{pmatrix},\label{compat1}\\
\begin{split}
&\begin{pmatrix}
\left(A(\mathbf{u}^{0}(0),0)\partial_{x}\mathbf{u}^{0}(0)+B(\mathbf{u}^{0}(0),0)\right)_{+}\\
\left(A(\mathbf{u}^{0}(L),L)\partial_{x}\mathbf{u}^{0}(L)+B(\mathbf{u}^{0}(L),L)\right)_{-}
\end{pmatrix}
=\\ & G'\begin{pmatrix}
\mathbf{u}_{+}^{0}(L)\\
\mathbf{u}_{-}^{0}(0)
\end{pmatrix} \begin{pmatrix}
\left(A(\mathbf{u}^{0}(L),L)\partial_{x}\mathbf{u}^{0}(L)+B(\mathbf{u}^{0}(L),L)\right)_{+}\\
\left(A(\mathbf{u}^{0}(0),0)\partial_{x}\mathbf{u}^{0}(0)+B(\mathbf{u}^{0}(0),0)\right)_{-}
\end{pmatrix}.
\end{split}
\label{compat2}
\end{gather}

Well-posedness of the system (\ref{system21}),(\ref{nonlocal}) for any initial condition $\mathbf{u}^{0}$ that satisfies the compatibility conditions (\ref{compat1}),(\ref{compat2})
is given by Li \cite{LiYu}
(see also \cite{Qin}), one has the following theorem:
\begin{thm}
For all $T>0$ there exist $C(T)>0$ and $\eta(T)>0$ such that, for every $\mathbf{u_{0}}\in C^{1}([0,L],\mathbb{R}^{n})$ satisfying the compatibility conditions (\ref{compat1}), (\ref{compat2}) and such that $\lvert \mathbf{u}_{0} \rvert_{1}\leq \eta$, the system (\ref{system21})-(\ref{nonlocal}) has a unique solution on $[0,T]\times[0,L]$ with initial condition $\mathbf{u_{0}}$. Moreover one \textcolor{black}{has}:
\begin{equation}
\lvert \mathbf{u}(t,\cdot)\rvert_{1}\leq C_{1}(T) \lvert \mathbf{u}(0,\cdot)\rvert_{1},\text{ }\forall t\in[0,T].
\label{Well-posedness}
\end{equation} 
\label{resultatLiRaoWang}
\end{thm}

\section{Context and previous results}
\paragraph{General hyperbolic system without source term}
The exponential stability of general strictly hyperbolic systems of the form (\ref{system21}) without source term, i.e. $B\equiv0$, has been mainly studied in the linear or non-linear case (see for instance \cite{CoronC1, CorondAndreaBastin, greenberg, slemrod, Tatsien, CoronBastinPI, CoronBV}) under various boundary conditions or boundary controls
(e.g. Proportional-integral control, dead beat control, single boundary control, etc.).
A large part of these studies has
been conducted using boundary conditions of the form (\ref{nonlocal}).
For such boundary conditions in non-linear systems the exponential stability depends on the 
topology \cite{CoronNguyen} 
and in particular that the stability in 
$H^{2}$ norm does not imply the stability in $C^{1}$ norm. In \cite{CoronNguyen} the authors also gave a sufficient condition for stability in 
the $W^{2,p}$ norm for $p\in[1,+\infty]$:
\begin{equation}
\rho_{p}(G'(0))<1,
\label{rho}
\end{equation} 
where $G$ is given in (\ref{nonlocal}) and the definition of $\rho_{p}$ is
\begin{equation}
\rho_{p}(M)=\inf(\lVert\Delta M \Delta^{-1}\rVert_{p}, \Delta\in D_{n}^{+}),\text{ }1\leq p\leq+\infty
\label{defrho}
\end{equation} 
where $\lVert\cdot\rVert_{p}$ is the usual $p$ norm for matrices and $D_{n}^{+}$ are the diagonal $n\times n$ matrices with positive eigenvalues.
\vspace{\baselineskip}\\
The case of the $C^{1}$ norm for systems with no source term has also been treated in \cite{CoronC1} by Jean-Michel Coron and Georges Bastin by a Lyapunov approach that inspired the first part of this paper. There, they proved the following sufficient condition for exponential stability through a Lyapunov approach:
\begin{equation}
\rho_{\infty}(G'(0))<1.
\end{equation} 
However the general case with a non-zero source term changes several things. As mentioned previously 
it implies that the steady-states $\mathbf{Y}^{*}$ are no longer necessarily uniform and as a direct consequence the matrix 
$A$ defined in 
(\ref{defA}) depend explicitly not only on $\mathbf{u}$ 
but also on $x$. In addition, 
there are some cases where, for any $G$, no basic quadratic $H^{2}$ Lyapunov function can be found (see for instance \cite{Coron1D} and in particular Proposition 5.12) or no basic $C^{1}$ Lyapunov function can be found, as shown later on.

\paragraph{General hyperbolic system with non-zero source term in the $H^{p}$ norm}
For general quasilinear hyperbolic systems with source term, also called inhomogeneous quasilinear hyperbolic systems, the analysis of the exponential stability is much less advanced and actual knowledge in the matter is still partial.
To our knowledge the exponential stability of such systems with non zero and non negligible source term was only treated in the framework of the $H^{p}$ norm for $p\in\mathbb{N}\setminus\{0,1\}$
and in \cite{Coron1D} (in Chapter 6) the authors find a sufficient (but $a$ $priori$ non-necessary) condition: exponential stability of the system (\ref{system21})--(\ref{Well-posedness}) for the $H^{p}$ norm where $p\geq2$ is achieved if there exists $Q\in C^{1}([0,L],D_{n}^{+})$ such that the two following conditions hold:
\begin{itemize}
\item (Interior condition) the matrix
\begin{equation}
-(Q\Lambda)'(x)+Q(x)M(\mathbf{0},x)+M(\mathbf{0},x)^{T}Q(x)^{T}
\label{intex}
 \end{equation} 
 is positive definite for all $x\in[0,L]$,
  \item (Boundary conditions) the matrix
 \begin{equation}
 \left(\begin{smallmatrix}
  \Lambda_{+}(L)Q_{+}(L)&0\\ 0 & -\Lambda_{-}(0)Q_{-}(0)
 \end{smallmatrix}\right)
-K^{T}\left(\begin{smallmatrix}
  \Lambda_{+}(0)Q_{+}(0)&0\\ 0 & -\Lambda_{-}(L)Q_{-}(L)     
      \end{smallmatrix}\right)K\\
      \label{boundex}
      \end{equation}
is positive semi-definite
\end{itemize}
\vspace{\baselineskip}
where $M(\mathbf{0},\cdot)=\frac{\partial B}{\partial \mathbf{u}}(\mathbf{0},\cdot)$ and $K=G'(\mathbf{0})$.\\

It has to be underlined that
with a non-zero source term in there does not always exist a simple quadratic Lyapunov function ensuring exponential stability for the $H^{p}$ norm whatever the boundary conditions are. Thus appears not only a boundary condition (\ref{boundex}) as in the previous paragraph but also an interior condition (\ref{intex}).

This phenomenon
is not specific to non-linear systems but also appears in linear systems: In \cite{CoronBastin22} for instance, the authors study a linear $2\times2$ system and found a necessary and sufficient condition for the existence of $Q$ such that (\ref{intex}) hold.
In general for linear hyperbolic systems the condition (\ref{intex}) also appears although it is only sufficient when $n>2$.
This is the consequence of the non-uniformity of the steady-states combined with non-identically vanishing zero order term even close to the steady states. 
If this phenomenon is not new, we will see however that the interior condition that appears for the $C^{1}$ norm is different from the condition that typically appears when studying Lyapunov functions for $H^{p}$ norms.
\vspace{\baselineskip}\\
Our contribution in this article is to deal with the exponential stability for the $C^{1}$ norm of such general hyperbolic systems with source term. This article intends to give a necessary and sufficient interior condition to the existence of a simple quadratic Lyapunov function ensuring exponential stability in the $C^{1}$ (and actually $C^{p}$) norm of the system and a sufficient condition on the boundary conditions.\\

\paragraph{Useful observations and notations}
Before going any further let us note that by definition of $B$ and as $\mathbf{Y}^{*}$ is a steady-state
\begin{equation}
B(\mathbf{0},x)=N(0)(F(\mathbf{Y}^{*})(\mathbf{Y}^{*}_{x})+D(\mathbf{Y}^{*}))=0.
\end{equation} 
Thus 
if we assume that
\textcolor{black}{$F$ and $Y^{*}$ are $C^{3}$ functions, then, from (\ref{defB}),}
 B is $C^{2}$ and there exists $\eta_{0}>0$ and $M\in C^{1}(\mathcal{B}_{\eta_{0}}\times[0,L],\mathcal{M}_{n}(\mathbb{R}))$, 
where $\mathcal{B}_{\eta_{0}}$ is the ball of radius $\eta_{0}$ in the space of continuous function endowed with the $L^{\infty}$ topology,
such that,
\begin{equation}
\begin{split}
B(\mathbf{u},x)=M(\mathbf{u},x)\mathbf{u},&
\\
\text{and therefore, }\frac{\partial B}{\partial  \mathbf{u}}(\mathbf{0},x)=M(\mathbf{0}&,x).
\label{M2}
\end{split}
\end{equation}
\textcolor{black}{Besides, $A$ is also a $C^{2}$ function and $\eta_{0}>0$ can be chosen small enough such that there exists $E
\in C^{2}(\mathcal{B}_{\eta_{0}}\times[0,L],\mathcal{M}_{n}(\mathbb{R}))$, 
satisfying (see \cite{Coron1D} in particular Lemma 6.7),
\begin{gather}
E(\mathbf{u},x)A(\mathbf{u},x)=\lambda(\mathbf{u},x)E(\mathbf{u},x) \text{ }\forall\text{  }(\mathbf{u},x)\in\mathcal{B}_{\eta_{0}}\times[0,L],
\label{defm}\\
\text{and } E(\mathbf{0},x)=Id,
\label{condId}
\end{gather} 
where $\lambda$ is a diagonal matrix, whose diagonal entries are the eigenvalues of $A(\mathbf{u},x)$.}\\

Also we introduce the following notations:
\begin{defn}
For a $C^{0}$ function $\mathbf{U}=(U_{1},...,U_{n})^{T}$ on $[0,L]$ we define the $C^{0}$ norm $\lvert  \mathbf{U}  \rvert_{0}$ by
\begin{equation}
\lvert \mathbf{U} \rvert_{0}: = \sup_{i}\left(\sup_{[0,L]} (\lvert U_{i}\rvert)\right).
\end{equation}
For a $C^{1}$ function $\mathbf{U}=(U_{1},...,U_{n})^{T}$ on $[0,L]$, we denote similarly the $C^{1}$ norm $\lvert \mathbf{U} \rvert_{1}$ by
\begin{equation}
\lvert \mathbf{U}\rvert_{1} := \lvert \mathbf{U} \rvert_{0}+\lvert \partial_{x} \mathbf{U}\rvert_{0}.
\end{equation} 
\end{defn}
In the following for a $C^{1}$ function $\mathbf{u}$ on $[0,T]\times[0,L]$, we will sometimes note for simplicity $\lvert \mathbf{u}\rvert_{0}$
instead of $\lvert \mathbf{u}(t,\cdot) \rvert_{0}$ and $\lvert \mathbf{u}\rvert_{1}$ instead of $\lvert \mathbf{u}(t,\cdot) \rvert_{1}$.
\vspace{\baselineskip}\\
We recall the definition of the exponential stability for the $C^{1}$ norm:
\begin{defn}
The steady state \textcolor{black}{$\mathbf{u}^{*}=0$} of the system (\ref{system21}),(\ref{nonlocal}) is exponentially stable for the $C^{1}$ norm if there exist $\gamma>0$, $\eta>0$, and $C>0$ 
such that for every $\mathbf{u}^{0}\in C^{1}([0,L])$ satisfaying the compatibility conditions (\ref{compat1}),(\ref{compat2}) and \textcolor{black}{$\lvert \mathbf{u}^{0}\rvert_{1}\leq\eta$}, the Cauchy problem 
(\ref{system21}),(\ref{nonlocal}),($\mathbf{u}(0,x)=\mathbf{u}^{0}$) has a unique $C^{1}$ solution and
\begin{equation}
\textcolor{black}{\lvert \mathbf{u}(t,\cdot)\rvert_{1}\leq C e^{-\gamma t}\lvert \mathbf{u}^{0} \rvert_{1}, \text{ }\forall t\in[0,+\infty[.}
\end{equation} 
\end{defn}

\begin{rmk}
Given our change of variable $\mathbf{Y}\rightarrow\mathbf{u}$, proving the exponential stability for the $C^{1}$ \textcolor{black}{norm of} the steady state $0$ of the system (\ref{system21}),(\ref{nonlocal}) is equivalent to proving the and to proving the exponential stability for the $C^{1}$ norm of the steady state $\mathbf{Y}^{*}$ of the system (\ref{system11}) and the associated boundary condition.
\end{rmk}


\begin{defn}
We call \textit{basic $C^{1}$ Lyapunov function} a function $V$ defined by
\begin{equation}
\begin{split}
V(\mathbf{U})=&\left| (\textcolor{black}{\sqrt{f_{1}}} U_{1},..., \textcolor{black}{\sqrt{f_{n}}}U_{n})^{T}\right|_{0}\\
+&\left| \left((\textcolor{black}{E(\mathbf{U},x)}(A(\mathbf{U},x)\mathbf{U}_{x}+B(\mathbf{U},x)))_{1}\textcolor{black}{\sqrt{f_{1}}},..., (\textcolor{black}{E(\mathbf{U},x)}(A(\mathbf{U},x)\mathbf{U}_{x}+B(\mathbf{U},x)))_{n}\textcolor{black}{\sqrt{f_{n}}}\right)^{T}\right|_{0},
\end{split}
\label{eqdefC1}
\end{equation} 
for some $(f_{1},...f_{n})\in C^{1}\left([0,L];\mathbb{R}^{*}_{+}\right)^{n}$, such that there exist
$\gamma>0$ and $\eta>0$
such that for any $T>0$ and any solution $\mathbf{u}$ of the system (\ref{system21})--(\ref{nonlocal})
with $\lvert \mathbf{u}^{0}\rvert_{1}\leq\eta$, 
\textcolor{black}{
\begin{equation}
 V(t)\leq V(t')e^{-\gamma (t-t')},\text{  }\forall\text{  }0\leq t'\leq t\leq T.
 \label{decroissance1}
\end{equation}}
Also, in that case, $(f_{1},...,f_{n})
$
are called coefficients inducing a basic $C^{1}$ Lyapunov function.
\vspace{\baselineskip}\\
\label{defC1}
\end{defn}
\begin{rmk}
Note from (\ref{system21}), that when $\mathbf{u}$ is a solution of the system (\ref{system21}), (\ref{nonlocal}), $V(\mathbf{u}(t,\cdot))$ becomes
\begin{equation}
V(\mathbf{u}(t,\cdot))=\left| (\textcolor{black}{\sqrt{f_{1}}} u_{1},..., \textcolor{black}{\sqrt{f_{n}}}u_{n})^{T}\right|_{0}+\left| (\textcolor{black}{E\mathbf{u}_{t}})_{1}\textcolor{black}{\sqrt{f_{1}}},..., (\textcolor{black}{E\mathbf{u}_{t}})_{n}\textcolor{black}{\sqrt{f_{n}}})^{T}\right|_{0},
\label{Vdef2}
\end{equation}
\textcolor{black}{where we denoted $E=E(\mathbf{u}(t,x),x)$ to lighten the notations.} The previous definition (\ref{eqdefC1}) is used so that $V$ is actually defined as function on $C^{1}([0,L])$ only and to underline that therefore, the function $V(\mathbf{u}):t\rightarrow V(\mathbf{u}(t,\cdot))$ does only depend on the state of the system at time $t$.
\textcolor{black}{Looking at \eqref{Vdef2}, one could wonder why we consider the components of $\mathbf{u}$ while we consider the components of $E\mathbf{u}_{t}$ for the derivative. 
The interest of considering $E\mathbf{u}_{t}$  instead of $\mathbf{u}_{t}$ is that $E$ diagonalizes $A$ and therefore when differentiating the Lyapunov function appears
$2(E\mathbf{u}_{t})_{n}(E(\mathbf{u})_{tt})_{n}=-\lambda_{n}(\mathbf{u},x)((E\mathbf{u}_{tx})_{n}^{2})$ and first order derivative terms, and there is no crossed term of second order derivative which would be impossible to bound with the $C^{1}$ norm
(the full computation is done in Appendix \ref{W2derivative}).
Differentiating $u^{2}_{n}$, though, gives $-\lambda_{n}(u_{n}^{2})_{x}-u_{n}((A-\lambda).\mathbf{u}_{x})_{n}$ and zero order derivative terms, and the second term is a cubic perturbation that can be bounded by the cube of the $C^{1}$ norm.
Nevertheless, the proof would work as well with $E\mathbf{u}$ instead of $\mathbf{u}$, but we consider $\mathbf{u}$ to keep the computations as simple as we can in the main proof (Section \ref{s4}). \textcolor{black}{Finally, we use in the definition \eqref{eqdefC1} the weights $\sqrt{f_{i}}$ instead of using directly the weights $f_{i}$ to be coherent with the existing definition of basic quadratic Lyapunov function for the $L^{2}$ norm introduced in \cite{CoronBastin22} (see in particular (34) ) for linear systems and to facilitate a potential comparison.} }

\end{rmk}
\textcolor{black}{
\begin{rmk}
Note also that, in Definition \ref{defC1}, the condition \eqref{decroissance1} is actually equivalent to the condition
\begin{equation}
\frac{dV(\mathbf{u})}{dt}\leq -\gamma V(\mathbf{u}),
\label{decroissance01}
\end{equation} 
in a distributional sense on $(0,T)$, where we say that $d\geq 0$ in a distributional sense on $(0,T)$ with $d\in \mathcal{D}'(0,T)$ when, for any $\phi\in C^{\infty}_{c}((0,T),\mathbb{R}^{+})$,
\begin{equation}
<d,\phi>\text{ }\geq\text{ }0.
\end{equation} 
\end{rmk}
}

Note that the existence of such basic $C^{1}$ Lyapunov function for a system guaranties the exponential stability of the system for the $C^{1}$ norm. More precisely we have the following proposition:
\begin{prop}
Let a quasilinear hyperbolic system be of the form (\ref{system21}),(\ref{nonlocal}), with $A$ and $B$ of class 
$C^{1}$ such that there exists a basic $C^{1}$ Lyapunov function, 
then the system is exponentially stable for the $C^{1}$ norm.
\label{proposition1}
\end{prop}

\begin{proof}[Proof of proposition \ref{proposition1}]
From Theorem \ref{resultatLiRaoWang}, let $T>0$ and $\mathbf{u}_{0}\in C^{1}([0,L],\mathbb{R}^{n})$ satisfying the compatibility conditions (\ref{compat1}) and such that $\lvert \mathbf{u}_{0}\rvert_{1}\textcolor{black}{\leq\min(\eta(T),\eta_{0}/C_{1}(T))}$, where $\eta(T)$ \textcolor{black}{and $C_{1}(T)$ are} given by Theorem \ref{resultatLiRaoWang} \textcolor{black}{and $\eta_{0}$ is given by \eqref{M2}--\eqref{condId}. 
From Theorem \ref{resultatLiRaoWang}} there exists a unique solution $\mathbf{u}\in C^{1}([0,T]\times[0,L])$. Suppose that $V$ is a basic $C^{1}$ Lyapunov function, induced by $(f_{1},...f_{n})$ and $\gamma$ and $\eta_{1}$
are the constants associated.
From its definition $V(\mathbf{u}(t,\cdot))$ is closely related to $\lvert \mathbf{u}(t,\cdot) \rvert_{1}$, indeed, using that for all $i\in\{1,n\}$, $f_{i}$ are positive and bounded on $[0,L]$, it is easy to see that there exists a constant $c_{2}>0$ such that
\textcolor{black}{
\begin{equation}
\frac{1}{c_{2}}(\lvert \mathbf{u}(t,\cdot)\rvert_{0}+\lvert \textcolor{black}{E\partial_{t}\mathbf{u}(t,\cdot)}\rvert_{0}) \leq V(\mathbf{u}(t,\cdot))\leq c_{2}(\lvert \mathbf{u}(t,\cdot)\rvert_{0}+\lvert \textcolor{black}{E\partial_{t}\mathbf{u}(t,\cdot)}\rvert_{0}).
\end{equation}
But as, from (\ref{Well-posedness}) and the assumption on $\lvert \mathbf{u}_{0}\rvert_{1}$, $|\mathbf{u}(t,\cdot)|_{1}\leq\eta_{0}$ for any $t\in[0,T]$. Thus from \eqref{defm}--\eqref{condId} there exists a constant $c_{1}$ depending only on $\eta_{0}$ and the system such that
\begin{equation}
\frac{1}{c_{1}}|\partial_{t}\mathbf{u}(t,\cdot)|_{0}\leq |E\partial_{t}\mathbf{u}(t,\cdot)|_{0}\leq c_{1}|\partial_{t}\mathbf{u}(t,\cdot)|_{0},
\end{equation} 
thus,} there exists $c_{0}>0$ such that
\begin{equation}
\frac{1}{c_{0}}(\lvert \mathbf{u}(t,\cdot)\rvert_{0}+\lvert \partial_{t}\mathbf{u}(t,\cdot)\rvert_{0}) \leq V(\mathbf{u}(t,\cdot))\leq c_{0}(\lvert \mathbf{u}(t,\cdot)\rvert_{0}+\lvert \partial_{t}\mathbf{u}(t,\cdot)\rvert_{0}).
\end{equation}
But observe that, as $\mathbf{u}$ is a solution of (\ref{system21}), there exists $\eta_{a}>0$ such that for $\lvert \mathbf{u}(t,\cdot) \rvert_{0}<\eta_{a}$
\begin{equation}
\lvert \partial_{t} \mathbf{u}(t,\cdot) \rvert_{0}\leq 2\sup_{i}\left(\lvert \Lambda_{i}\rvert_{0}\right)\lvert \partial_{x} \mathbf{u}(t,\cdot)\rvert_{0}+2\sup_{i,j}\left(\lvert M_{ij}(\mathbf{0},\cdot)\rvert_{0}\right)\lvert \mathbf{u}(t,\cdot)\rvert_{0},\\
\end{equation}
and similarly
\begin{equation}
\lvert \partial_{x} \mathbf{u}(t,\cdot) \rvert_{0}\leq \frac{2}{\inf_{i,x\in[0,L]}\left(\Lambda_{i}(x)\right)}\left(\lvert\partial_{t} \mathbf{u}(t,\cdot)\rvert_{0}+\sup_{i,j}\left(\lvert M_{ij}(\mathbf{0},\cdot)\rvert_{0}\right)\lvert \mathbf{u}(t,\cdot)\rvert_{0}\right),\\
\end{equation}
which implies that there exists $c>0$ constant such that for $\lvert \mathbf{u}(t,\cdot) \rvert_{0}<\eta_{a}$
\begin{equation}
\frac{1}{c}\lvert \mathbf{u}(t,\cdot) \rvert_{1} \leq V(\mathbf{u}) \leq c\lvert \mathbf{u} (t,\cdot) \rvert_{1}.
\label{encadrement0}
\end{equation} 

Let $T\in\mathbb{R}^{*}_{+}$, with $T>0$
 and $T$ large enough such that $c^{2}e^{-\gamma T}<\frac{1}{2}$. From (\ref{decroissance1}), for all solution $\mathbf{u}$ such that $\lvert \mathbf{u}^{0}\rvert_{1}<\min(\eta(T),\eta_{1}, \eta_{a}/C(T))$ where $C(T)$ is defined in (\ref{Well-posedness}),
\begin{equation}
V(\mathbf{u},T)\leq V(\mathbf{u},0)e^{-\gamma T}.
\end{equation} 

Now, using (\ref{encadrement0}) we get
\begin{equation}
\lvert \mathbf{u}(T,\cdot) \rvert_{1} \leq \lvert \mathbf{u} (0,\cdot) \rvert_{1} c^{2} e^{-\gamma T},
\end{equation} 
And from the hypothesis on $T$
\begin{equation}
\lvert \mathbf{u}(T,\cdot) \rvert_{1} \leq \frac{1}{2}\lvert \mathbf{u}(0,\cdot) \rvert_{1},
\end{equation} 
and this imply that $\mathbf{u}$ is defined on $[0,+\infty)$ and that we can find $C$ and $\gamma_{1}$ such that 
\begin{equation}
\lvert \mathbf{u}(t,0)-\mathbf{u}^{*}\rvert_{1}\leq C e^{-\gamma t}\lvert \mathbf{u^{0}-\mathbf{u}^{*}} \rvert_{1}, \text{ }\forall t\in[0,+\infty[,
\end{equation}
which gives the exponential stability and concludes the proof.
\end{proof}

\section{Main results}

The aim of this article is to show the following results:

\begin{thm}
Let a quasilinear hyperbolic system be of the form (\ref{system21}), (\ref{nonlocal}), with $A$ and $B$ of class 
$C^{1}$, $\Lambda$ defined as in (\ref{Lambda}) and $M$ as in (\ref{M2}). Let assume that the two following properties hold
\begin{enumerate}
\item (Interior condition) the system
\begin{equation}
\Lambda_{i}f_{i}'\leq -2\left(-M_{ii}(0,x)f_{i} + \sum\limits_{k=1,k\neq i}^{n}\lvert M_{ik}(0,x)\rvert \frac{f_{i}^{3/2}}{\sqrt{f_{k}}} \right),
\label{cond111}
\end{equation} 
admits a solution $(f_{1},...,f_{n})$ on $[0,L]$ such that for all $i\in[1,n]$, $f_{i}>0$,

\item (Boundary conditions) there exists a diagonal matrix $\Delta$ with positive coefficients such that
\begin{equation}
\lVert \Delta G'(0)\Delta^{-1}\rVert_{\infty}<\frac{\inf_{i}\left(\frac{f_{i}(d_{i})}{\Delta_{i}^{2}}\right)}{\sup_{i}\left(\frac{f_{i}(L-d_{i})}{\Delta_{i}^{2}}\right)},
\label{condauxbords}
\end{equation} 
where 
$d_{i}=L$ if $\Lambda_{i}>0$ and $d_{i}=0$ otherwise.

\end{enumerate}
Then there exists a basic $C^{1}$ Lyapunov function for the system (\ref{system21}), (\ref{nonlocal}).
\label{resultat1}
\end{thm}

\begin{rmk}
Note that when $M\equiv0$ we recover the result found in \cite{CoronC1} in the absence of source term: the interior condition is always verified 
by any positive constant functions $(f_{1},...,f_{n})$ and when choosing $f_{i}=\Delta_{i}^{2}$ the boundary condition reduces to the existence of $\Delta\in D_{+}^{n}$ such that 
$\lVert \Delta G'(0)\Delta^{-1}\rVert_{\infty}<1$ which is equivalent to $\rho_{\infty}(G'(0))<1$.\\

Note also that the existence of a solution $(f_{1},...f_{n})$ with $f_{i}>0$ on $[0,L]$ for all $i\in\{1,...,n\}$ for the system
\begin{equation}
f_{i}'=-\frac{2}{\Lambda_{i}}\left(- M_{ii}(0,x)f_{i} +\sum\limits_{k=1,k\neq i}^{n}\lvert M_{ik}(0,x)\rvert \frac{f_{i}^{3/2}}{\sqrt{f_{k}}} \right)
\end{equation} 
is also a sufficient interior condition as it obviously implies the existence of a solution with positive components for (\ref{cond111}).
\end{rmk}
\vspace{\baselineskip}
Moreover, we show in the following Theorem that condition (\ref{cond111}) is also necessary in order to ensure the existence of a basic $C^{1}$ Lyapunov function.
\begin{thm}
Let a quasilinear hyperbolic system be of the form (\ref{system21}) with $A$ and $B$ of class $C^{3}$, there exists a control of the form (\ref{nonlocal}) such that there exists a basic $C^{1}$ Lyapunov function for the system (\ref{system21}),(\ref{nonlocal}) if and only if
\begin{equation}
\Lambda_{i}f_{i}'\leq -2\left(\sum\limits_{k=1,k\neq i}^{n}\lvert M_{ik}(0,x)\rvert \frac{f_{i}^{3/2}}{\sqrt{f_{k}}} - M_{ii}(0,x)f_{i} \right),
\label{cond122}
\end{equation} 
admits a solution $(f_{1},...,f_{n})$ on $[0,L]$ such that for all $i\in[1,n]$, $f_{i}>0$.
\label{resultat2}
\end{thm}

\begin{rmk}
Note that Theorem \ref{resultat2} illustrates the sharpness of (\ref{cond111}) by showing that it is a necessary condition.
This is not trivial as, to our knowledge, there is no similar condition for the $H^{p}$ norm when $n>2$ yet.
Note also that we have not imposed anything on the initial values of the $(f_{1},...,f_{n})$
but we see from Theorem \ref{resultat1} and (\ref{condauxbords}) that the more liberty we give them, the more restrictive the condition on the boundary (\ref{condauxbords}) might become.
\end{rmk}
The proof of these two results is given in the next section. 
\vspace{\baselineskip}
\section{$\mathbf{C^{1}}$ Lyapunov stability of $n\times n$ quasilinear hyperbolic system}

In this Section we shall prove Theorem \ref{resultat1} and Theorem \ref{resultat2}. We will first start by proving the following Lemma \textcolor{black}{which will be useful 
for finding the interior condition in the proof of Theorem \ref{resultat1} and for proving Theorem \ref{resultat2}}:

\begin{lem}
Let $(a_{i},b_{ij})_{(i,j)\in\llbracket1,n\rrbracket^{2}}\in C([0,L],\mathbb{R})^{n}\times C([0,L],\mathbb{R})^{n^{2}}$,
\vspace{\baselineskip}\\
If
\begin{equation}
\begin{split}
(i)\text{ }\exists p_{1}\in\mathbb{N}^{*}:
\sum\limits_{i=1}^{n} \left(a_{i}(x)y_{i}^{2p} + \sum\limits_{j=1}^{n} b_{ij}(x)y_{i}^{2p-1}y_{j}\right) > 0,\text{ }\forall p>p_{1},\forall y\in\mathbb{R}^{n}\setminus \{0\},\forall x\in[0,L],
\end{split}
\label{lemma1}
\end{equation} 
then
\begin{equation}
(ii)\text{ }a_{i}(x)\geq \sum\limits_{j=1,j\neq i}^{n}\lvert b_{ij}(x)\rvert-b_{ii}(x),\text{ }\forall i\in[1,n],\forall x\in[0,L].
\end{equation} 
And if
\begin{equation}
(iii)\text{ }a_{i}(x)> \sum\limits_{j=1,j\neq i}^{n}\lvert b_{ij}(x)\rvert-b_{ii}(x),\forall i\in[1,n], \forall x\in[0,L],
\label{hypo3}
\end{equation} 
then $(i)$ holds.
\label{lemma}
\end{lem}
\begin{proof}[Proof of Lemma \ref{lemma}]
We start with $(i)\Rightarrow(ii)$.
Let $x\in[0,L]$, let $\textcolor{black}{i_{1}}\in[1,n]$, assuming $(i)$ is true for all $y\in\mathbb{R}^{n}\setminus \{0\}$, we take $m\in\mathbb{N}^{*}$, and define $y_{\textcolor{black}{i_{1}}}:=1$, $y_{j}:=-\sgn(b_{\textcolor{black}{i_{1}}j})m/(m+1)$ for $j\neq \textcolor{black}{i_{1}}$.
Then as (\ref{lemma1}) is true \textcolor{black}{there exists $p_{1}\in\mathbb{N}^{*}$ such that
\begin{equation}
\begin{split}
&\sum\limits_{i=1,i\neq i_{1}}^{n} \left(a_{i}(x)y_{i}^{2p} + \sum\limits_{j=1}^{n} b_{ij}(x)y_{i}^{2p-1}y_{j}\right)\\
&+a_{i_{1}}(x)+b_{i_{1}i_{1}}+\sum\limits_{j=1,j\neq i_{1}}^{n} b_{i_{1}j}(x)y_{j}
> 0,\text{ }\forall p>p_{1},\forall x\in[0,L].
\end{split}
\end{equation} 
Note that for any $i\neq i_{1}$, $\lim_{p\rightarrow +\infty}|y_{i}|^{2p}=0$. Thus}, by letting $p\rightarrow + \infty$ one gets
\begin{equation}
a_{\textcolor{black}{i_{1}}}(x)+b_{\textcolor{black}{i_{1}}\textcolor{black}{i_{1}}}(x) \geq \frac{m}{m+1} \sum\limits_{j=1,j\neq \textcolor{black}{i_{1}}}^{n} \lvert b_{\textcolor{black}{i_{1}}j}(x) \rvert,\text{ }\forall x\in[0,L].
\end{equation} 
Hence, as it is true for all $m\in\mathbb{N}^{*}$, letting $m\rightarrow +\infty$
\begin{equation}
a_{\textcolor{black}{i_{1}}}(x)+b_{\textcolor{black}{i_{1}}\textcolor{black}{i_{1}}}(x) \geq \sum\limits_{j=1,j\neq \textcolor{black}{i_{1}}}^{n} \lvert b_{\textcolor{black}{i_{1}}j}(x) \rvert,\text{ }\forall x\in[0,L].
\end{equation} 
\textcolor{black}{This can be done for any $i_{1}\in[1,n]$, which concludes}
$(i)\Rightarrow(ii)$.
\vspace{\baselineskip}\\
Now let us prove that $(iii)\Rightarrow(i)$. First of all observe that we can suppose without loss of generality that $\forall i\in[1,n]$, $b_{ii}:=0$: one just has to redefine $a_{i}:=a_{i}+b_{ii}$. 
Then by (\ref{hypo3}), $a_{i} > \sum\limits_{j=1}^{n} \lvert b_{ij} \rvert$, $\forall i\in[1,n]$, then let us define:
\begin{equation}
d_{i}(x):=a_{i}(x) - \sum\limits_{k=1}^{n} \lvert b_{ik}(x) \rvert,
\label{di}
\end{equation} 
then $d_{i}$ is $C^{0}$ and positive on $[0,L]$. We denote by
\begin{equation}
 d_{i}^{(0)}:=\inf_{[0,L]}(d_{i})=\min_{[0,L]}(d_{i})>0.
\end{equation} 
Now, let $y\in\mathbb{R}^{n}\setminus\{0\}$,
we can select $i_{1}$ such that
\begin{equation}
 \lvert y_{i_{1}}\rvert =\max_{i\in[1,n]}(\lvert y_{i} \rvert),
\end{equation} 
thus $y_{i_{1}}\neq 0$ and proving (\ref{lemma1}) is equivalent to proving that there exists $p_{1}\in\mathbb{N}^{*}$ such that for all $p>p_{1}$,
\begin{equation}
\sum\limits_{i=1}^{n} \left(a_{i}(x)\left|\frac{y_{i}}{y_{i_{1}}}\right|^{2p} + \sum\limits_{k=1}^{n}  b_{ik}(x) \left(\frac{y_{i}}{y_{i_{1}}}\right)^{2p-1}\frac{y_{k}}{y_{i_{1}}}\right)> 0,\text{ }\forall x\in[0,L].
\label{tp1}
\end{equation} 
Denoting $z_{i}=y_{i}/y_{i_{1}}$, (\ref{tp1}) becomes
\begin{equation}
I:=\sum\limits_{i=1}^{n} \left(a_{i}z_{i}^{2p} + \sum\limits_{k=1}^{n}  b_{ik} z_{i}^{2p-1}z_{k}\right)>0,\text{ on }[0,L].
\end{equation} 
Using (\ref{di}) we know that
\begin{equation}
I=\sum\limits_{i=1}^{n} d_{i}z_{i}^{2p} + \sum\limits_{k=1}^{n} \lvert b_{ik} \rvert z_{i}^{2p} + \sum\limits_{k=1}^{n}  b_{ik} z_{i}^{2p-1}z_{k}.
\end{equation} 
By definition for $i=i_{1}$, $ \lvert z_{i_{1}}\rvert=1$, and for $i\neq i_{1}$, $\lvert z_{k} \rvert\leq1$, therefore
\begin{equation}
d_{i_{1}}z_{i_{1}}^{2p} + \sum\limits_{k=1}^{n} \lvert b_{i_{1}k} \rvert z_{i_{1}}^{2p} + \sum\limits_{k=1}^{n}  b_{i_{1}k} z_{i_{1}}^{2p-1}z_{k}\geq d_{i_{1}}\geq d_{i_{1}}^{(0)}.
\end{equation} 
Therefore
\begin{equation}
\begin{split}
 I&\geq d_{i_{1}}^{(0)} +\sum\limits_{i=1, i\neq i_{1}}^{n}\left( d_{i}z_{i}^{2p} + \sum\limits_{k=1}^{n} \lvert b_{ik} \rvert \lvert z_{i}\rvert^{2p} - \sum\limits_{k=1}^{n}  \lvert b_{ik}\rvert \lvert z_{i}\rvert^{2p-1}\right),\\
 &= d_{i_{1}}^{(0)} +\sum\limits_{i=1, i\neq i_{1}}^{n} \left(d_{i}z_{i}^{2p} - \sum\limits_{k=1}^{n} \lvert b_{ik} \rvert (1-\lvert z_{i}\rvert) \lvert z_{i}\rvert^{2p-1}\right).
 \end{split}
\end{equation} 
We introduce
\begin{equation}
 g:z\mapsto g(z)= -(1-z) z^{2p-1},
\end{equation} 
We know that g is $C^{1}$ on [0,1] and admits a minimum on $[0,1]$ at $z=1-\frac{1}{2p}$, as one can check that
\begin{equation}
 g'(z)= (2pz-(2p-1)) z^{2p-2}.
\end{equation} 
Therefore
\begin{equation}
 I\geq d_{i_{1}}^{(0)} - \frac{1}{2p} \sum\limits_{i=1,i\neq i_{1}}^{n} \sum\limits_{k=1}^{n} \lvert b_{ik}(x) \rvert,
\end{equation} 
and this is true for all $x\in[0,L]$.
Let us point out that there exists $p_{1}>0$ such that
\begin{equation}
\frac{1}{2p}\sum\limits_{i=1,i\neq i_{1}}^{n} \sum\limits_{k=1}^{n} \lvert b_{ik} \rvert_{0} < d_{i_{1}}^{(0)},\text{ }\forall p>p_{1}.
\end{equation} 
Here $p_{1}$ is a constant and does not depend on $x$.
Hence we can conclude that $I>0,\text{ }\forall p>p_{1},\forall x\in[0,L],\forall y\in\mathbb{R}^n$. Therefore (\ref{lemma1}) holds.
\end{proof}
\vspace{\baselineskip}
Now let us prove Theorem \ref{resultat1}.
\begin{proof}[Proof of Theorem \ref{resultat1}]
Let $T\in\mathbb{R}_{+}^{*}$. Let assume that $A$ and $B$ are of class $C^{2}$, and
let
 $\mathbf{u}$ be a 
$C^{2}$ solution of system (\ref{system21}),(\ref{nonlocal}) such that $\lvert\mathbf{u}^{0}\rvert_{1}\leq \varepsilon$. Such solution exists for $\varepsilon$ small enough and $\mathbf{u}_{0}\in C^{2}([0,L],\mathbb{R}^{n})$ which verifies the compatiblity conditions (\ref{compat1}) 
(see \cite{Coron1D} in particular Theorem 4.21). We suppose here a $C^{2}$ regularity for technical reason but the final estimate will not depend on the $C^{2}$ norm and will be also true by density for $A$ and $B$ of class $C^{1}$ and for $\mathbf{u}$ a $C^{1}$ solution.
 Recall that $\lambda_{i}$ are the eigenvalues of $A$ as defined in (\ref{defA}). 
We denote $s_{i}:=\sgn(\lambda_{i}(\mathbf{u},x))$ which only depends on $i$ from the hypothesis of non-vanishing eigenvalues and the continuity of $A$.
We define:
\begin{equation}
  W_{1,p}:=\left(\int_{0}^{L} \sum\limits_{i=1}^{n} f_{i}(x)^{p} u_{i}^{2p} e^{-2p\mu s_{i} x} dx \right)^{1/2p},
 \label{W1}
\end{equation} 
with $p\in\mathbb{N}^{*}$, and $f_{i}>0$ on $[0,L]$ to be determined. Clearly $W_{1,p}>0$ for $\mathbf{u}\neq0$, and $W_{1,p}=0$ when $\mathbf{u}\equiv0$. If we differentiate $W_{1,p}$ with respect to time along the $C^{2}$ trajectories, we have
\begin{equation}
\begin{split}
\frac{dW_{1,p}}{dt}=& W_{1,p}^{1-2p}\int_{0}^{L} \sum\limits_{i=1}^{n} f_{i}(x)^{p} u_{i}^{2p-1}\left[-\sum\limits_{k=1}^{n}a_{ik}(\mathbf{u},x)u_{kx}\right.
\\ & \left.-\sum\limits_{k=1}^{n}M_{ik}(\mathbf{u},x)u_{k}\right] e^{-2p\mu s_{i}x} dx,
\end{split}
\end{equation} 
where $(a_{ij})_{(i,j)\in[1,n]^2}=A$ and $M$ is defined in (\ref{M2}). We know that the $a_{ij}$ are $C^{2}$ and from (\ref{M2}) that $a_{ij}(0,\cdot)=\delta_{i,j}\Lambda_{i}(\cdot)$. Here $\delta_{i,j}$ stands for the Kronecker delta. Hence
\begin{equation}
a_{ij}(\mathbf{u},\cdot)=\delta_{i,j}\Lambda_{i}(\cdot)+V_{ij}.\mathbf{u},
\end{equation}
where $V_{ij}$ are $C^{1}$.
Therefore using integration by parts 
\begin{equation}
\begin{split}
 \frac{dW_{1,p}}{dt}= & -\frac{W_{1,p}^{1-2p}}{2p}\left[\sum\limits_{i=1}^{n}\lambda_{i} f_{i}(x)^{p} u_{i}^{2p} e^{-2p\mu s_{i}x}\right]_{0}^{L} 
 \\  &-W_{1,p}^{1-2p}\int_{0}^{L}\sum\limits_{i=1}^{n}  f_{i}(x)^{p} u_{i}^{2p-1} \left[\left(\sum\limits_{k=1}^{n}M_{ik}u_{k}\right)\right.
 +\left.\sum\limits_{k=1}^{n}(V_{ik}(\mathbf{u},x).\mathbf{u})u_{kx} \right] e^{-2p\mu s_{i}x} dx 
 \\  &+\frac{W_{1,p}^{1-2p}}{2}\int_{0}^{L}\sum\limits_{i=1}^{n} \left(\lambda_{i}(\mathbf{u},x)  f_{i}(x)^{p-1}f'_{i} u_{i}^{2p}
+\frac{d}{dx}\frac{(\lambda_{i}(\mathbf{u},x))}{p} f_{i}(x)^{p}u_{i}^{2p}\right) e^{-2p\mu s_{i}x} dx 
 \\  &-\mu W_{1,p}^{1-2p} \int_{0}^{L}\sum\limits_{i=1}^{n} \lvert\lambda_{i}\rvert   f_{i}^{p}u_{i}^{2p}e^{-2p\mu s_{i}x} dx.
 \end{split}
 \label{68}
\end{equation} 
We denote
\begin{equation}
\mathrm{I}_{2}:=\frac{W_{1,p}^{1-2p}}{2p}\left[\sum\limits_{i=1}^{n}\lambda_{i}   f_{i}(x)^{p} u_{i}^{2p} e^{-2p\mu s_{i}  x }\right]_{0}^{L},
\label{I2}
\end{equation} 
and 
\begin{equation}
\begin{split}
\mathrm{I}_{3}:=& W_{1,p}^{1-2p}\int_{0}^{L} \sum\limits_{i=1}^{n}   f_{i}(x)^{p} u_{i}^{2p-1} \left(\sum\limits_{k=1}^{n}M_{ik}u_{k}\right) e^{-2p\mu s_{i}x} dx
\\ & -\frac{W_{1,p}^{1-2p}}{2}\int_{0}^{L}\sum\limits_{i=1}^{n} \lambda_{i}(\mathbf{u},x)   f_{i}(x)^{p-1}f'_{i} u_{i}^{2p}e^{-2p\mu s_{i} x} dx.
\label{I3}
\end{split}
\end{equation} 
We supposed that
 $\lvert\mathbf{u}^{0}\rvert_{1}\leq \varepsilon$, where $\varepsilon>0$ can be chosen arbitrarily small but, of course, independent of $p$. 
From (\ref{Well-posedness}) and denoting $\eta=C_{1}(T)\varepsilon$ we have: $\lvert \mathbf{u}\rvert_{0}\leq\eta$. 
Choosing $\varepsilon$ sufficiently small is thus equivalent to choosing $\eta$ sufficiently small, so we will rather choose $\eta$ in the following and this choice of $\eta$ will always be independent of $p$. 
Besides, observe that there exists $\eta_{1}>0$ sufficiently small 
such that for all $\mathbf{u}_{0}\in C^{0}([0,L],\mathbb{R}^{n})$ such that $\lvert \mathbf{u}\rvert_{0}\leq\eta_{1}$
\begin{equation}
\min_{x\in[0,L]}\left(\min_{i\in[1,n]}\left(\lvert \lambda_{i}\left(\mathbf{u},x\right)\rvert\right)\right)\geq \min_{x\in[0,L]} \left(\min_{i\in[1,n]}\left(\frac{\lvert \Lambda_{i}(x)\rvert}{2}\right)\right).
\label{lambda1}
\end{equation} 
Recall that $\Lambda=\lambda(0,\cdot)$ and is defined in (\ref{Lambda}). As $[0,L]$ is a closed segment, and the $\left|\Lambda_{i}\right|$ are strictly positive continuous functions we can define the positive constant $\alpha_{0}:=\min_{x\in[0,L]} \left(\min_{i\in[1,n]}\left(\lvert \Lambda_{i}(x)\rvert/2\right)\right)>0$. We suppose from now on that $\eta<
\eta_{1}$. 
Therefore from (\ref{68}), (\ref{I2}), (\ref{I3}) and (\ref{lambda1})
\begin{equation}
\begin{split}
 \frac{dW_{1,p}}{dt}\leq&-I_{2}-\mu \alpha_{0} W_{1,p}-I_{3} 
 \\ & -W_{1,p}^{1-2p}\int_{0}^{L}\sum\limits_{i=1}^{n}  f_{i}^{p} u_{i}^{2p-1}\left(\sum\limits_{j=1}^{n}(V_{ik}(\mathbf{u},x).\mathbf{u} )u_{kx}\right)e^{-2p\mu s_{i} x }\\
 & +\frac{W_{1,p}^{1-2p}}{2p}\int_{0}^{L}\sum\limits_{i=1}^{n} \left(\frac{\partial\lambda_{i}}{\partial  \mathbf{u} }.\mathbf{u}_{x}+\partial_{x}\lambda_{i}\right)   f_{i}(x)^{p}u_{i}^{2p}e^{-2p\mu s_{i} x} dx.
 \label{decomp1}
\end{split}
\end{equation}
We now estimate the two last terms, starting by the last one. The $\lambda_{i}$ are $C^{2}$ and in particular $C^{1}$ in $\mathbf{u}$ therefore
\begin{equation}
\begin{split}
\frac{W_{1,p}^{1-2p}}{2p}&\int_{0}^{L}\sum\limits_{i=1}^{n} (\frac{\partial\lambda}{\partial  \mathbf{u} }.\mathbf{u}_{x}+\partial_{x}\lambda_{i})   f_{i}(x)^{p}u_{i}^{2p}e^{-2p\mu s_{i} x} dx
 \\ &\leq \frac{C_{1}}{2p} W_{1,p}+\frac{C_{2}}{2p}W_{1,p}\lvert \mathbf{u} \rvert_{1},
 \end{split}
\end{equation} 
where $C_{1}$ and $C_{2}$ are constants that depend on $\eta$ and the system but are independent from $p$ and $\mathbf{u}$ provided that $\lvert\mathbf{u}\rvert_{1}<\eta$. Besides we have
\begin{equation}
W_{1,p}^{1-2p}\int_{0}^{L}\sum\limits_{i=1}^{n}  f_{i}^{p} u_{i}^{2p-1}(\sum\limits_{j=1}^{n}(V_{ik}(\mathbf{u},x).\mathbf{u}) u_{kx})e^{-2p\mu x s_{i}} dx\leq C_{3} W_{1,p}\lvert \mathbf{u}\rvert_{1}.
\end{equation}
where $C_{3}$ is a constant that does not depend on on $p$ and $\mathbf{u}$.
Therefore
(\ref{decomp1}) can be written as
\begin{equation}
\frac{dW_{1,p}}{dt}\leq-I_{2}-I_{3}-(\mu \alpha_{0}- \frac{C_{1}}{2p}) W_{1,p} +(\frac{C_{2}}{2p}+C_{3})W_{1,p}\lvert\mathbf{u}\rvert_{1}.
\label{decomp2}
\end{equation} 
As $\alpha_{0}>0$, it is easy to see that there exists $p_{1}\in \mathbb{N}^{*}$ such that $\forall p\geq p_{1}$
\begin{equation}
 \frac{dW_{1,p}}{dt}\leq-I_{2}-I_{3}-\frac{\mu \alpha_{0}}{2} W_{1,p} +C_{4}W_{1,p}\lvert\mathbf{u}\rvert_{1}.
\label{decomp3}
\end{equation} 
Here $p_{1}$ depends only on $\alpha_{0}$ and $\eta$, while $C_{4}$ does not depend on $p$ and $\mathbf{u}$. Before going any further, we see here that if we can manage to prove that $I_{2}>0$ and $I_{3}\text{ }\textcolor{black}{\geq}\text{ }0$ we may be able to conclude to the existence of a Lyapunov function that looks like a $L^{2p}$ norm where $p$ can be as large as we want and therefore we start to see the forecoming basic $C^{1}$ Lyapunov function. We are now left with studying $I_{2}$ and $I_{3}$ which will correspond respectively to the boundary condition and the interior condition we mentioned in Section 2 and in Theorem \ref{resultat1}.
\vspace{\baselineskip}\\
Let us first deal with $I_{3}$:
\begin{equation}
\begin{split}
\mathrm{I}_{3}=&W_{1,p}^{1-2p}\int_{0}^{L} \sum\limits_{i=1}^{n} \left( f_{i}^{p}u_{i}^{2p-1}\left(\sum\limits_{k=1}^{n}M_{ik}u_{k}\right)-\frac{\lambda_{i} f'_{i}}{2} f_{i}^{p-1}u_{i}^{2p}\right)e^{-2p\mu s_{i} x} dx.
\end{split}
\end{equation} 
%
Let suppose that the system (\ref{cond111}) admits a positive solution $(g_{1},...g_{n})$ on $[0,L]$, which is the interior condition. 
Then we can write this as
\begin{equation}
-\Lambda_{i}g_{i}'= 2\left(\sum\limits_{k=1,k\neq i}^{n}\lvert M_{ik}(0,x)\rvert \frac{g_{i}^{3/2}}{\sqrt{g_{k}}} - M_{ii}(0,x)g_{i} \right)+h_{i},
\label{gi}
\end{equation}
where $h_{i}$ are non-negative functions. By continuity (see for instance \cite{parametre}, in particular Theorem 2.1 in Chapter 5) there exists $\sigma_{1}>0$ such that for all $\sigma\in[0,\sigma_{1}]$ there exists a unique solution to
\begin{equation}
\begin{split}
-\Lambda_{i}f_{i}'= &2\left(\sum\limits_{k=1,k\neq i}^{n}\lvert M_{ik}(0,x)\rvert \frac{f_{i}^{3/2}}{\sqrt{f_{k}}} - M_{ii}(0,x)f_{i} \right)+h_{i}+\sigma,\\
f_{i}(0)=&g_{i}(0).
\end{split}
\end{equation}
We denote $(f_{1,\sigma},...f_{n,\sigma})$ this solution, which is continuous with $\sigma$. Therefore there exists $\sigma_{2}\in(0,\sigma_{1}]$ such that for all $i\in[1,n]$, and all  $\sigma\in(0,\sigma_{2}]$, $f_{i,\sigma}>0$, on $[0,L]$ and
\begin{equation}
\begin{split}
&-\Lambda_{i}f_{i,\sigma}' > 2\left(\sum\limits_{k=1,k\neq i}^{n}\lvert M_{ik}(0,x)\rvert \frac{f_{i,\sigma}^{3/2}}{\sqrt{f_{k,\sigma}}} - M_{ii}(0,x)f_{i,\sigma} \right).
\end{split}
\end{equation}
We choose now $f_{i}:=f_{i,\sigma}$ where $\sigma\in(0,\sigma_{2}]$.
As 
$M$ and $\lambda$ are continuous in $\mathbf{u}$, there exists $\eta_{2}>0$ such that for $\lvert \mathbf{u}\rvert_{0}<\eta_{2}$
\begin{equation}
-\lambda_{i}(\mathbf{u},x) \frac{f'_{i}}{f_{i}}>2\sum\limits_{k=1,k\neq i}^{n}\lvert M_{ik}(\mathbf{u},x)\rvert\sqrt{\frac{f_{i}}{f_{k}}} - 2 M_{ii}(\mathbf{u},x).
\label{cond10}
\end{equation} 
Therefore from Lemma \ref{lemma}
\begin{equation}
\sum\limits_{i=1}^{n}\left( -\frac{\lambda_{i} f'_{i}}{2f_{i}}y_{i}^{2p} + \sum\limits_{j=1}^{n}M_{ik}\frac{\sqrt{f_{i} }}{\sqrt{f_{k}}}y_{k}y_{i}^{2p-1}\right)>0,\text{ }\forall\text{ } \textcolor{black}{\mathbf{y}}=(y_{i})_{i\in[1,n]}\in\mathbb{R}^{n}\setminus\{0\},  
\end{equation} 
\textcolor{black}{applying this for $(y_{i})_{i\in[1,n]}=\textcolor{black}{(\sqrt{f_{i}}u_{i}))_{i\in[1,n]}}$, it} implies that
\begin{equation}
W_{1,p}^{1-2p}\int_{0}^{L} \sum\limits_{i=1}^{n} \left(-\textcolor{black}{\frac{\lambda_{i} f'_{i}}{2}f_{i}^{p-1}}u_{i}^{2p}+\sum\limits_{k=1}^{n}M_{ik}u_{k}\textcolor{black}{f_{i}^{p}}u_{i}^{2p-1} \right)dx\text{ }\textcolor{black}{\geq}\text{ }0.
\end{equation}
Therefore by continuity, there exists a $\mu_{1}>0$ such that $\forall\mu\in[0,\mu_{1}]$
\begin{equation}
\mathrm{I}_{3}=W_{1,p}^{1-2p}\int_{0}^{L} \sum\limits_{i=1}^{n} \left(-\textcolor{black}{\frac{\lambda_{i} f'_{i}}{2}f_{i}^{p-1}}u_{i}^{2p}+\textcolor{black}{f_{i}^{p}}u_{i}^{2p-1}\left(\sum\limits_{k=1}^{n}M_{ik}u_{k}\right) \right)e^{-2p\mu  s_{i} x} dx>0.
\end{equation} 

Now let us deal with $I_{2}$, which will lead to the boundary condition. 
Recall that
\begin{equation}
\begin{split}
\mathrm{I}_{2}=&\frac{W_{1,p}^{1-2p}}{2p}\left[\sum\limits_{i=1}^{n}\lambda_{i}(\mathbf{u}(t,L),L)   f_{i}(L)^{p} u_{i}^{2p}(t,L) e^{-2p\mu s_{i} L}\right.
\\ & \left.-\sum\limits_{i=1}^{n}\lambda_{i}(\mathbf{u}(t,0),0)   f_{i}(0)^{p} u_{i}^{2p}(t,0)\right].
\end{split}
\end{equation} 
Recall that $m$ is the integer such that $\Lambda_{i}>0,$ for all $i\leq m $ and $\Lambda_{i}<0,$ for all $ i>m$, we have
\begin{equation}
\begin{split}
\mathrm{I}_{2}=&\frac{W_{1,p}^{1-2p}}{2p}\left(\sum\limits_{i=1}^{m}\lvert\lambda_{i}(\mathbf{u}(t,L),L)\rvert   f_{i}(L)^{p} u_{i}^{2p}(t,L) e^{-2p\mu L}\right.
\\&\left.-\sum\limits_{i=1}^{m}\lvert\lambda_{i}(\mathbf{u}(t,0),0)\rvert   f_{i}(0)^{p} u_{i}^{2p}(t,0) \right.\\
&\left. -\sum\limits_{i=m+1}^{n}\lvert\lambda_{i}(\mathbf{u}(t,L),L)\rvert   f_{i}(L)^{p} u_{i}^{2p}(t,L) e^{2p\mu L}\right.\\
&\left.+\sum\limits_{i=m+1}^{n}\lvert\lambda_{i}(\mathbf{u}(t,0),0)\rvert   f_{i}(0)^{p} u_{i}^{2p}(t,0)\right),
\end{split}
\end{equation} 
We denote $K:=G'(0)$ and we know that under assumption (\ref{condauxbords}) there exists $\Delta=\left(\Delta_{1},...,\Delta_{n}\right)^{T}\in(\mathbb{R}_{+}^{*})^{n}$ such that 
\begin{equation}
\theta:=\sup_{\lVert \xi \rVert_{\infty}\leq1}(\sup_{i}(\lvert \sum\limits_{j=1}^{n}(\Delta_{i} K_{ij}\Delta_{j}^{-1})\xi_{j} \rvert))<\frac{\inf_{i}\left(\frac{g_{i}(d_{i})}{\Delta_{i}^{2}}\right)}{\sup_{i}\left(\frac{g_{i}(L-d_{i})}{\Delta_{i}^{2}}\right)}.
\end{equation} 
where $(g_{i})_{i\in[1,n]}$ denote the positive solution of (\ref{cond111}) introduced previously in (\ref{gi}). Note that we have in fact $\theta=\sup_{i}(\sum\limits_{i=0}^{n}\lvert K_{ij}\rvert \frac{\Delta_{i}}{\Delta_{j}})$.
Let:
\begin{gather}
 \xi_{i}=\Delta_{i}u_{i}(t,L)\text{ for }i\in[1,m],\\
 \xi_{i}=\Delta_{i}u_{i}(t,0)\text{ for }i\in[m+1,n].
\end{gather} 
From (\ref{nonlocal}) and using the fact that $G$ is $C^1$, we have
\begin{equation}
\begin{pmatrix}
u_{+}(t,0)\\
u_{-}(t,L)
\end{pmatrix}
= K\begin{pmatrix}
u_{+}(t,L)\\
u_{-}(t,0)
\end{pmatrix}
+o\left(\left| \begin{pmatrix}
u_{+}(t,L)\\
u_{-}(t,0)
\end{pmatrix}
\right|\right),
\end{equation} 
where $o(x)$ refers to a function such that $o(x)/\lvert x \rvert$ tends to $0$ when $\lvert\mathbf{u}\rvert_{0}$ tends to $0$.
Thus we get
\begin{equation}
\begin{split}
\mathrm{I}_{2}=&\frac{W_{1,p}^{1-2p}}{2p}\left(\sum\limits_{i=1}^{m} \lambda_{i}(\mathbf{u}(t,L),L) \frac{f_{i}(L)^{p}}{\Delta_{i}^{2p}} (u_{i}(t,L)\Delta_{i})^{2p} e^{-2p\mu L}\right.
\\ &\left. +\sum\limits_{i=m+1}^{n} \lvert \lambda_{i}(\mathbf{u}(t,0),0)\rvert \frac{f_{i}(0)^{p}}{\Delta_{i}^{2p}}(u_{i}(t,0)\Delta_{i})^{2p}\right.
\\ &\left.-\sum\limits_{i=1}^{m}\lambda_{i}(\mathbf{u}(t,0),0) \frac{f_{i}(0)^{p}}{\Delta_{i}^{2p}} (\sum\limits_{k=1}^{n}K_{ik}\xi_{k}(t)\frac{\Delta_{i}}{\Delta_{k}}\textcolor{black}{+o(\mathbf{\xi})})^{2p}\right.
\\ &\left.-\sum\limits_{i=m+1}^{n}\left|\lambda_{i}(\mathbf{u}(t,L),L)\right| \frac{f_{i}(L)^{p}}{\Delta_{i}^{2p}} (\sum\limits_{k=1}^{n}K_{ik}\xi_{k}(t)\frac{\Delta_{i}}{\Delta_{k}}\textcolor{black}{+o(\mathbf{\xi})})^{2p}e^{2p\mu L}\right)
\end{split}
\end{equation}
As the $\lambda_{i}$ are $C^{1}$ in $\mathbf{u}$ we have
\begin{equation}
\begin{split}
\mathrm{I}_{2}=&\frac{W_{1,p}^{1-2p}}{2p}\left(\sum\limits_{i=1}^{m} (\Lambda_{i}(L)\textcolor{black}{+O(\mathbf{\xi})}) \frac{f_{i}(L)^{p}}{\Delta_{i}^{2p}} (u_{i}(t,L)\Delta_{i})^{2p} e^{-2p\mu L}\right.
\\ &\left. +\sum\limits_{i=m+1}^{n} \lvert (\Lambda_{i}(0)\textcolor{black}{+O(\mathbf{\xi})})\rvert \frac{f_{i}(0)^{p}}{\Delta_{i}^{2p}}(u_{i}(t,0)\Delta_{i})^{2p}\right.
\\ &\left.-\sum\limits_{i=1}^{m}(\Lambda_{i}(0)\textcolor{black}{+O(\mathbf{\xi})})\frac{f_{i}(0)^{p}}{\Delta_{i}^{2p}} (\sum\limits_{k=1}^{n}K_{ik}\xi_{k}(t)\frac{\Delta_{i}}{\Delta_{k}}\textcolor{black}{+o(\mathbf{\xi})})^{2p}\right.
\\ &\left.-\sum\limits_{i=m+1}^{n}\lvert (\Lambda_{i}(L)\textcolor{black}{+O(\mathbf{\xi})})\rvert \frac{f_{i}(L)^{p}}{\Delta_{i}^{2p}} (\sum\limits_{k=1}^{n}K_{ik}\xi_{k}(t)\frac{\Delta_{i}}{\Delta_{k}}\textcolor{black}{+o(\mathbf{\xi})})^{2p}e^{2p\mu L}\right)
\end{split}
\end{equation}
\textcolor{black}{where $O(x)$ refers to a function such that $O(x)/|x|$ is bounded when $\lvert\mathbf{u}\rvert_{0}$ tends to $0$.} Now let $t\in[0,T]$, there exists $i_{0}$ such that $\max_{i}(\xi_{i}^{2}(t))=\xi_{i_{0}}^{2}$, to simplify the notations we introduce $d_{i}$ such that $d_{i}=L$ for $i\leq m$ and $d_{i}=0$ for $i\geq m+1$.
Then \textcolor{black}{there exists a constant $C>0$ independant of $\mathbf{u}$ and $p$ such that}
\begin{equation}
\begin{split}
 \mathrm{I}_{2}\geq& \frac{W_{1,p}^{1-2p}}{2p}((\lvert\Lambda_{i_{0}}(d_{i_{0}})\rvert\textcolor{black}{-C|\xi_{i_{0}}|}) \frac{f_{i_{0}}^{p}(d_{i_{0}})}{\Delta_{i_{0}}^{2p}} \xi_{i_{0}}^{2p}(t) e^{-2p\mu d_{i_{0}}}
 \\ & -\sum\limits_{i=1}^{n} (\lvert\Lambda_{i}(L-d_{i})\rvert\textcolor{black}{+C|\xi_{i_{0}}|})\frac{f_{i}^{p}(L-d_{i})}{\Delta_{i}^{2p}} (\theta \textcolor{black}{+l(\xi_{i_{0}})})^{2p}\xi_{i_{0}}^{2p}e^{2p\mu (L-d_{i})})
 \end{split}
\end{equation}  
\textcolor{black}{where $l$ is a continuous and positive function which satisfies $l(0)=0$.}
thus
\begin{equation}
\begin{split}
 \mathrm{I}_{2}\geq& \frac{W_{1,p}^{1-2p}}{2p}((\lvert\Lambda_{i_{0}}(d_{i_{0}})\rvert\textcolor{black}{-C|\xi_{i_{0}}|}) \frac{f_{i_{0}}^{p}(d_{i_{0}})}{\Delta_{i_{0}}^{2p}} \xi_{i_{0}}^{2p}(t) e^{-2p\mu d_{i_{0}}}
 \\ &  -n \sup_{i\in[1,n]}\left((\lvert\Lambda_{i}(L-d_{i})\rvert\textcolor{black}{+C|\xi_{i_{0}}|})\frac{f_{i}^{p}(L-d_{i})}{\Delta_{i}^{2p}}e^{2p\mu (L- d_{i})}\right)(\theta\textcolor{black}{+l(\xi_{i_{0}})})^{2p}\xi_{i_{0}}^{2p})
 \end{split}
 \label{I211}
\end{equation}  
Now, from (\ref{condauxbords})
 we have
\begin{equation}
\theta^{2}<\frac{\inf_{i}\left(\frac{g_{i}(d_{i})}{\Delta_{i}^{2}}\right)}{\sup_{i}\left(\frac{g_{i}(L-d_{i})}{\Delta_{i}^{2}}\right)},
\end{equation}
where $(g_{i})_{i\in[1,n]}$ still denote the positive solution of (\ref{cond111}). Remark that we set earlier $f_{i}:=f_{i,\sigma}$ where $\sigma\in(0,\sigma_{2}]$ and can be chosen arbitrary small, and recall that the functions $f_{i,\sigma}$ are continuous in $\sigma$ on this neighbourhood of $0$. Therefore there exists $\sigma\in(0,\sigma_{2}]$ such that
\begin{equation}
\theta^{2}<\frac{\inf_{i}\left(\frac{f_{i}(d_{i})}{\Delta_{i}^{2}}\right)}{\sup_{i}\left(\frac{f_{i}(L-d_{i})}{\Delta_{i}^{2}}\right)}.
\end{equation}
But as the inequality is strict, there exist by continuity \textcolor{black}{$\eta_{3}\in(0,\eta_{2})$}, $p_{3}>0$ and $\mu_{3}$ such that \textcolor{black}{for all $\lvert \mathbf{u}\rvert_{0}<\eta_{3}$ and $p>p_{3}$}
\begin{equation}
\left(\theta\textcolor{black}{+l(\xi_{i0})}\right)^{2}<\left(\frac{\inf_{i}\lvert\Lambda_{i}(d_{i})\rvert\textcolor{black}{-C|\xi_{i_{0}}|}}{n \left(\sup_{i}\lvert\Lambda_{i}(L-d_{i})\rvert\textcolor{black}{+C|\xi_{i_{0}}|}\right)}\right)^{1/p}\frac{\inf_{i}\left(\frac{f_{i}(d_{i})}{\Delta_{i}^{2}}\right)}{\sup_{i}\left(\frac{f_{i}(L-d_{i})}{\Delta_{i}^{2}}\right)}e^{-4\mu L},\text{ }\forall \mu\in[0,\mu_{3}], \forall p\geq p_{3}.
 \label{I201}
\end{equation} 
Therefore from (\ref{I201}) and (\ref{I211}) 
 $\mathrm{I}_{2}>0$.
We can conclude that there exist $p_{4}$ and $\mu>0$ 
\begin{equation}
\frac{dW_{1,p}}{dt}\leq-\frac{\mu \alpha_{0}}{2} W_{1,p} +C_{6}W_{1,p}\lvert\mathbf{u}\rvert_{1},\text{ }\forall p\geq p_{4}.
\label{IW1}
\end{equation} 
\vspace{\baselineskip}\\
We now have our first estimate and we have seen appear both an interior condition and a boundary condition that explains the conditions that appear in Theorem \ref{resultat1}. Yet there remains a potentially non-negative term in $\lvert\mathbf{u}\rvert_{1}$ and the function we considered in (\ref{W1}) does not have the form of a basic $C^{1}$ Lyapunov function.
The last step is now to convert $W_{1,p}$ in a basic $C^{1}$ Lyapunov function.
Defining
\begin{equation}
 W_{2,p}=\left(\int_{0}^{L} \sum\limits_{i=1}^{n}   f_{i}(x)^{p}\textcolor{black}{(E\mathbf{u_{t}})_{i}^{2p}} e^{-2p\mu s_{i}x} dx \right)^{1/2p},
 \label{W2}
\end{equation} 
where \textcolor{black}{$E=E(\mathbf{u}(t,x),x)$ is given by 
\eqref{defm},} and proceeding the same way and observing that, for $C^{2}$ solutions,
\begin{equation}
 \mathbf{u}_{tt}+A(\mathbf{u},x)\mathbf{u}_{tx}+\left[\frac{\partial A}{\partial  \mathbf{u}}(\mathbf{u},x).\mathbf{u}_{t}\right]\mathbf{u}_{x}+ \frac{\partial B}{\partial  \mathbf{u}}(\mathbf{u},x)\mathbf{u}_{t}=0,
\end{equation} 
where $\partial A/\partial \mathbf{u}.\mathbf{u}_{t}$ refers to the matrix with coefficients $\sum\limits_{k=1}^{n}\partial A_{ij}/\partial \mathbf{u}_{k}(\mathbf{u},x).\partial_{t}\mathbf{u}_{k}(t,x)$, 
we can obtain similarly
\begin{equation}
\frac{dW_{2,p}}{dt}\leq-\frac{\mu \alpha_{0}}{2} W_{2,p} +C_{7}W_{2,p}\lvert\mathbf{u}\rvert_{1}.
\end{equation} 
In order to avoid overloading this article, the proof -which is very similar to the proof of (\ref{IW1})-  is given in the Appendix (see \ref{W2derivative}).

Now let us define $W_{p}:=W_{1,p}+W_{2,p}$, there exists $\eta_{4}>0$ (independent of $p$), $\mu>0$, $C$ (independent of $p$ and $\mathbf{u}$), and $p_{5}$ such that, with $\lvert \mathbf{u} \rvert_{1}<\eta_{4}$,
\begin{equation}
\frac{dW_{p}}{dt}\leq-\frac{\mu \alpha_{0}}{2} W_{p} +C W_{p}\lvert \mathbf{u} \rvert_{1},\text{ }\forall p\geq p_{5}.
\label{IWp}
\end{equation}
Here we see that this estimate does not depend on the $C^{2}$ norm of the solution $\mathbf{u}$ 
and of the $C^{2}$ norms of $A$ and $B$
and is therefore also true by density for solutions that are only of class $C^{1}$ and for $A$ and $B$ also only $C^{1}$. To be fully rigourous, this statement assumes the well-posedness of the system (\ref{system21}), (\ref{nonlocal}), ($\mathbf{u}=\mathbf{u}^{0}$) in $W^{1,\infty}$ when $u_{0}\in W^{1,\infty}([0,L])$, but such well posedness is true (see \cite{LiYu}). 
We choose such $\eta,\mu,p_{5}$, and we define our basic $C^{1}$ Lyapunov function candidate
\begin{equation}
\begin{split}
 V&:=\lvert \textcolor{black}{\sqrt{f_{1}}}\textcolor{black}{u_{1}} e^{-\mu x\frac{\lambda_{1}}{\lvert\lambda_{1}\rvert}},...,  \textcolor{black}{\sqrt{f_{n}}} \textcolor{black}{u_{n}} e^{-\mu x\frac{\lambda_{n}}{\lvert\lambda_{n}\rvert}}\rvert_{0}\\
 &+\lvert \sqrt{f_{1}} \textcolor{black}{(E\mathbf{u_{t}})_{1}} e^{-\mu x\frac{\lambda_{1}}{\lvert\lambda_{1}\rvert}},..., \sqrt{f_{n}} \textcolor{black}{(E\mathbf{u_{t}})_{n}} e^{-\mu x\frac{\lambda_{n}}{\lvert\lambda_{n}\rvert}}\rvert_{0}.
 \end{split}
\end{equation} 
Similarly to the method used in \cite{CoronC1} we can first choose $\eta_{5}<\min(\eta_{1},\eta_{2},\eta_{3},\eta_{4})$ such that for all $\eta<\eta_{5}$ 
\begin{equation}
 \lvert \mathbf{u} \rvert_{1}<\frac{\mu\alpha_{0}}{4 C}.
 \label{eta1}
\end{equation} 
\begin{rmk}
Recall that $\lvert \mathbf{u}\rvert_{1}\leq \eta$ and that for convenience we are choosing $\eta$ the bound on $\lvert \mathbf{u} \rvert_{1}$ instead of choosing 
$\varepsilon$, the bound on $\lvert \mathbf{u}^{0} \rvert_{1}$, but from (\ref{Well-posedness}) it is equivalent. Hence the previous only means choosing $\varepsilon_{2}>0$ small enough, and such that for all $\varepsilon<\varepsilon_{2}$
\begin{equation}
\lvert \mathbf{u}(0,\cdot) \rvert_{1}<\frac{\mu\alpha_{0}}{4 C_{1}(T) C},
\end{equation} 
where $C_{1}(T)$ is the constant defined in (\ref{Well-posedness}).
\end{rmk}
\vspace{\baselineskip}
Therefore from (\ref{IWp}) and (\ref{eta1})
\begin{equation}
\frac{dW_{p}}{dt}\leq-\frac{\mu \alpha_{0}}{4} W_{p}(t),\text{ }\forall p\geq p_{5}.
\end{equation} 
\textcolor{black}{
Thus, using Gronwall Lemma, one has, for any $p\geq p_{5}$ and any $0\leq t'\leq t\leq T$,
\begin{equation}
W_{p}(t)\leq W_{p}(t')e^{-\frac{\mu \alpha_{0}}{4}(t-t')}.
\end{equation} 
Then, by definitions of $W_{p}$ and $V$
\begin{gather}
\lim_{p\to +\infty} W_{p}(t)=V^{\textcolor{black}{2}}(t),\text{ }\forall t\in[0,T],
\label{convergence}
\end{gather}
Therefore
\begin{equation}
V(t)\leq V(t')e^{-\frac{\mu \alpha_{0}}{\textcolor{black}{8}}(t-t')},\text{  }\forall\text{  }0\leq t'\leq t\leq T.
\end{equation} 
}
Therefore $V$ is a basic $C^{1}$ Lyapunov function with the associated constants $\gamma=\frac{\mu \alpha_{0}}{\textcolor{black}{8}}$ and $\eta=\eta_{5}$.
\end{proof}

\paragraph{Proof of Theorem \ref{resultat2}}
\begin{proof}[Proof]
The sufficient way is simply proven by 
using
 Theorem \ref{resultat1} with $G\equiv0$ for instance. 
We are left with proving the necessary way. Let us suppose that there exists a basic $C^{1}$ Lyapunov function $V$ induced by coefficients $(f_{1},...f_{n})$ and $\gamma$ and $\eta_{1}$
the constants associated such that $V$ is a Lyapunov function for all $\mathbf{u}$ smooth solution that satisfies the compatibility conditions and such that $\lvert \mathbf{u}\rvert_{0}<\eta_{1}$.
Suppose now by contradiction that the system (\ref{cond122}) does not admit a solution $(g_{1},...,g_{n})$ on $[0,L]$ such that for all $i\in[1,n]$, $g_{i}>0$. 
Then there exist $x_{0}\in[0,L]$ and $i_{0}\in[1,n]$ such that
\begin{equation}
-\Lambda_{i_{0}}(x_{0})f'_{i_{0}}(x_{0})<2\sum\limits_{k=1,k\neq i_{0}}^{n}\lvert M_{i_{0}k}(0,x_{0})\rvert\frac{f_{i_{0}}^{3/2}(x_{0})}{\sqrt{f_{k}(x_{0})}} - 2 M_{i_{0}i_{0}}(0,x_{0})f_{i_{0}}(x_{0}),
\label{lemma1cont0}
\end{equation} 
as, if not, $(f_{1},...f_{n})$ would be a solution on $[0,L]$ to (\ref{cond122}) with $f_{i}>0$, for all $i\in[1,n]$. We can rewrite (\ref{lemma1cont0}) simply as
\begin{equation}
-\sum\limits_{k=1,k\neq i_{0}}^{n}\left|M_{i_{0}k}(0,x_{0})\right|\frac{\sqrt{f_{i_{0}}(x_{0})}}{\sqrt{f_{k}(x_{0})}} -\frac{\Lambda_{i_{0}}(x_{0}) f'_{i_{0}}(x_{0})}{2f_{i_{0}}(x_{0})}+M_{i_{0}i_{0}}(0,x_{0})<0.
\label{lemma1cont}
\end{equation} 

For simplicity we can assume without losing any generality that $i_{0}=1$. By continuity there exists $\varepsilon>0$ such that (\ref{lemma1cont}) is true on $[x_{0}-\varepsilon,x_{0}+\varepsilon]\cap[0,L]$. 
We actually can suppose without loss of generality that $x_{0}\in(0,L)$ and that $[x_{0}-\varepsilon,x_{0}+\varepsilon]\subset(0,L)$. 
\vspace{\baselineskip}\\
Then we take $u_{1}^{0}\in(-\eta_{2},\eta_{2})$ positive, where $\eta_{2}$ is a positive constant arbitrary so far,
and define the vector $\mathbf{u}^{0}$ by
\begin{gather}
u_{i}^{0}:=-u_{1}^{0}\left(1-\frac{1}{k}\right)\sgn(M_{1i}(0,x_{0})), \text{ }\forall i\neq1,
\label{y01}
\end{gather}
where 
$k\in\mathbb{N}^{*}$ is arbitrary and $\sgn(0)=0$.
As the system is strictly hyperbolic, $\min(\lvert\lambda_{i}(x_{0})\rvert)$ is achieved at most for two $i\in[1,n]$. 
If so, we denote $i_{0}$ and $i_{1}$ the corresponding index, and if $i_{0}\neq 1$ and $i_{1}\neq 1$ we can redefine $u_{i_{1}}^{0}$ by
\begin{equation}
u_{i_{0}}^{0}:=-u_{1}^{0}\left(1-\frac{1}{k_{2}}\right)\sgn(M_{1i_{0}}(0,x_{0})),
\label{redefine}
\end{equation}
where $k_{2}\in \mathbb{N}^{*}$ with $k_{2}>k$. The goal of this redefinition is that in both cases we can choose $k$ large enough so that 
\begin{equation}
(i\neq i_{0})\Rightarrow \left|\frac{u_{i}^{0}}{\lambda_{i}(x_{0})}\right|<\left|\frac{u_{i_{0}}^{0}}{\lambda_{i_{0}}(x_{0})}\right|.
\label{ineqi0}
\end{equation} 
We now define the initial condition by
\begin{gather}
u_{i}(0,x):=\frac{u_{i}^{0}}{m}\chi(x)\frac{e^{-m(x-x_{0})-c}}{\lambda_{i}(x)\sqrt{ f_{i}(x)}},
\label{y1}
\end{gather} 
where $\chi:[0,L]\rightarrow \mathbb{R}$ is a $C^{\infty}$ function with compact support in $(0,L)$ to be determined, such that $\lvert \chi\rvert_{0}$ is independent of $m\in\mathbb{N}^{*}$ which will be set large enough and $c$ is a constant independent from $m$, also to be determined.
In order to simplify the notations we will suppose here that $\lambda_{1}>0$,
otherwise one only needs to replace $e^{-m(x-x_{0})-c}$ by $e^{-\sgn(\lambda_{1})(m(x-x_{0})+c)}$ to obtain the same result. Note here that the compatibility conditions are satisfied for this initial condition as the function and its derivatives vanish on the boundaries. From (\ref{y01}) and (\ref{y1}), we can choose $\eta_{2}$ small enough and independent of $m$ such that $\lvert \mathbf{u}(0,\cdot) \rvert_{1}<\eta_{1}$. Well-posedness of the system guaranties the existence and uniqueness of a solution $y$ to the system (\ref{system21}),(\ref{nonlocal}) with such initial condition (see Theorem \ref{resultatLiRaoWang}). For simplicity we will conduct the proof assuming that the system is linear, (i.e. $\lambda_{i}(\mathbf{u},\cdot)=\Lambda_{i}$,  $a_{ij}(\mathbf{u},\cdot)=\delta_{ij}\Lambda_{i}(\cdot)$, \textcolor{black}{$E(\mathbf{u},\cdot)=Id$,} and $M(\mathbf{u},\cdot)=M(0,\cdot)$) 
 although it is also not needed and is only to simplify the computations. A way to transform the proof for non-linear system is given in the Appendix (see \ref{adapting}).

Before going any further and selecting $\chi$, we shall first give the idea and explain our strategy.
We want to select $\chi$ such that
$\lvert \sqrt{f_{1}} \partial_{t}u_{1}(0,\cdot),..., \sqrt{f_{n}} \partial_{t}u_{n}(0,\cdot)  \rvert_{0}$ is achieved for  $i=1$ and $x=x_{1}$ close to $x_{0}$ and only for such $i$ and $x_{1}$. 
We also want $d/dt\lvert \sqrt{f_{1}} u_{1}(0,\cdot),..., \sqrt{f_{n}} u_{n}(0,\cdot)  \rvert_{0}(0)$ to exist and to be $O\left(d/dt\left(\sqrt{f_{1}(x_{0})} \partial_{t}u_{1}(0,x_{0})/m\right)\right)$ such that 
$dV/dt(0)$ will exist and its sign will be given by the sign of $\sqrt{ f_{1}(x_{0})} \partial_{tt}^{2}u_{1}(0,x_{0})$. Then we will show that this sign is positive.
\vspace{\baselineskip}\\
Now let us select $\chi$ in order to achieve these goals. 
Rephrasing our first objective, we want that for all $i\neq 1$
\begin{equation}
\sqrt{f_{1}(x_{1})}\left|\lambda_{1}(x_{1})\partial_{x}u_{1}(0,x_{1})+\sum\limits_{j=1}^{n}M_{1j}u_{j}(0,x_{1})\right|>\sup_{x\in[0,L]}\left(\sqrt{f_{i}(x)}\left|\lambda_{i}(x)\partial_{x}u_{i}(0,x)+\sum\limits_{j=1}^{n}M_{ij}u_{j}(0,x)\right|\right),
\end{equation} 
while the maximum of $\sqrt{f_{1}}\lvert\lambda_{1}\partial_{x}u_{1}(0,\cdot)+\sum\limits_{j=1}^{n}M_{1j}u_{j}(0,\cdot)\rvert$ is achieved only in $x_{1}$, close to $x_{0}$.\\
We search $\chi$ under the form
\begin{equation}
\chi=\phi(m(x-x_{0})),
\label{phi}
\end{equation} 
where $\phi$ is a positive $C^{\infty}$ function with compact support. 
And we search $\chi$ such that all the $\left|\sqrt{f_{i}}\partial_{t} u_{i}(0,\cdot)\right|$ admit their maximum at a single point in a small neighbourhood of $x_{0}$. In that case note that from (\ref{y01}) 
we would indeed get that for $m$ large enough $\lvert \sqrt{ f_{1}} \partial_{t}u_{1}(0,\cdot),..., \sqrt{f_{n}} \partial_{t}u_{n}(0,\cdot)  \rvert_{0}$ is attained for  $i=1$ only 
and at a single point close to $x_{0}$. This will be shown rigorously later (see (\ref{atteint})). 
Now let us look at $\sqrt{f_{i}}\partial_{t} u_{i}(0,\cdot)$
\begin{equation}
\begin{split}
\sqrt{f_{i}}\partial_{t} u_{i}(0,x)=& -u_{i}^{0}e^{-m(x-x_{0})-c}\left[-\chi(x)+\frac{\chi'(x)}{m} +\frac{\chi(x)\lambda_{i}\sqrt{f_{i}}}{m}\left(\frac{1}{\lambda_{i}\sqrt{f_{i}}}\right)'\right.\\ 
&\left.+\frac{1}{m}\sum\limits_{j=1}^{n} M_{ij}\left(\frac{u_{j}^{0}}{u_{i}^{0}}\right)\left(\sqrt{\frac{f_{i}}{f_{j}}}\right)\frac{1}{\lambda_{j}}\chi(x)\right].
\end{split}
\label{quantite}
\end{equation} 
Using (\ref{phi}) and a change of variable $y=m(x-x_{0})$, (\ref{quantite}) becomes
\begin{equation}
\begin{split}
\sqrt{f_{i}}\partial_{t} u_{i}(0,x)=& -u_{i}^{0}e^{-y-c}\left[\vphantom{\sum\limits_{0}^{n}}-\phi(y)+\phi'(y)\right.\\
&\left.+\left(\frac{g_{i}(\frac{y}{m}+x_{0})}{m}+\sum\limits_{j=0}^{n}\frac{f_{ij}(\frac{y}{m}+x_{0})}{m}\right)\phi(y)\right],
\end{split}
\label{quantite2}
\end{equation} 
where $g_{i}$ and $f_{ij}$ are $C^{2}$ bounded functions on $[0,L]$ independent of $m$. 
This comes from the fact that $A$ and $B$ are of class $C^{3}$. This hypothesis, that does not appear in Theorem \ref{resultat1}, is used to apply the implicit function theorem later on (see (\ref{encadrement}) and (\ref{V2})).
Theorem \ref{resultat2} might also be proven with lower hypothesis on the regularity $A$ and $B$, however in most physical case $A$ and $B$ are $C^{3}$ even when the solutions of the system are much less regular.
We can see that the coefficients of 
the equation (\ref{quantite2}) in $\phi$ and $\phi'$ depend on $m$ and are close to be constant for large $m$. 
One can show that there exists a function $\psi_{0}$ such that $\psi_{0}\in C_{c}^{3}((-1,1))$, such that $\lvert (\psi_{0}(y)-\psi_{0}'(y))e^{-y} \rvert$ has a unique maximum on $[-1,1]$ which is $1$, and such that the second derivative of $\lvert (\psi_{0}(y)-\psi_{0}'(y))e^{-y} \rvert$ does not vanish in this point, i.e. there exists a unique $y_{1}\in(-1,1)$ such that
\begin{gather}
\left|\psi_{0}(y)-\psi_{0}'(y)\right|e^{-y}<1=\left|\psi_{0}(y_{1})-\psi_{0}'(y_{1})\right|e^{-y_{1}},\text{ }\forall y\in[-1,1]\setminus \{y_{1}\},
\label{nondegenerate0}\\
\left(\left|\psi_{0}-\psi_{0}'\right|e^{-Id}\right)''(y_{1})\neq 0.
\label{nondegenerate}
\end{gather} 
The existence of this function $\psi_{0}$ is shown in the Appendix (see \ref{psi0existence}).
We set $\phi:y\rightarrow \psi_{0}(y+y_{1})$ and $c=y_{1}$. Therefore
\begin{equation}
e^{-y-c}\left[-\phi(y)+\phi'(y)\right]=(-\psi_{0}(y+y_{1})+\psi_{0}'(y+y_{1}))e^{-(y+y_{1})},
\end{equation} 
which has a maximum absolute value for $y=0$ with value equal to $1$.
Hence, there exists $m_{1}>0$ such that for all $m>m_{1}$ and all $i\in[1,n]$
\begin{gather}
\exists!x_{i}\in[x_{0}-\varepsilon,x_{0}+\varepsilon]: \lvert\sqrt{f_{i}(x_{i})} \partial_{t}u_{i}(0,x_{i})\rvert=\sup_{[0,L]}(\lvert \sqrt{f_{i}} \partial_{t}u_{i}(0,\cdot)\rvert),
\label{encadrement}\\
-u_{i}^{0}-\frac{C_{i}}{m}\lvert u_{i}^{0}\rvert\leq \sqrt{f_{i}(x_{i})} \partial_{t}u_{i}(0,x_{i})\leq-u_{i}^{0}+\frac{C_{i}}{m}\lvert u_{i}^{0}\rvert,
\label{estimation}
\end{gather} 
where $C_{i}$ are constants that do not depend on $m$. The unicity in (\ref{encadrement}) comes from the condition (\ref{nondegenerate}) which ensures that the maximum stays unique when the function is slightly perturbated. We can actually replace $C_{i}$ by $C=\max_{i}(C_{i})>0$. Therefore, there exists $m_{2}>m_{1}$ such that for all $m>m_{2}$
and $i\in[2,n]$
\begin{equation}
\sup_{[0,L]}(\lvert \sqrt{f_{i}} \partial_{t}u_{i}(0,\cdot)\rvert)\leq(1-\frac{1}{k})\left(1+\frac{C}{m}\right)u_{1}^{0}<u_{1}^{0}\left(1-\frac{C}{m}\right)\leq\sup_{[0,L]}(\lvert \sqrt{f_{1}} \partial_{t}u_{1}(0,\cdot)\rvert).
\label{atteint}
\end{equation} 
Hence, as we announced earlier,
\begin{equation}
\lvert \sqrt{f_{1}} \partial_{t}u_{1}(0,\cdot),..., \sqrt{f_{n}} \partial_{t}u_{n}(0,\cdot)  \rvert_{0}=\lvert \sqrt{f_{i}} \partial_{t}u_{i}(0,x)\rvert \Longleftrightarrow i=1,x=x_{1}.
\label{imply1}
\end{equation} 
Hence, as $u_{1}^{0}>0$ and from (\ref{encadrement}) and (\ref{estimation}),
\begin{equation}
\lvert \sqrt{f_{1}} \partial_{t}u_{1}(0,\cdot),..., \sqrt{f_{n}} \partial_{t}u_{n}(0,\cdot)  \rvert_{0}=-\sqrt{ f_{1}(x_{1})} \partial_{t}u_{1}(0,x_{1}).
\label{inequality}
\end{equation} 
Therefore, 
as the maximum is unique and the inequality of (\ref{atteint}) is strict, and from (\ref{nondegenerate}) and the implicit function theorem, provided that $m$ is large enough there exist $t_{1}>0$ and $x_{a}\in C^{1}([0,t_{1}];[0,L])$ such that
\begin{equation}
\begin{split}
&\lvert \sqrt{ f_{1}} \partial_{t}u_{1}(t,\cdot),..., \sqrt{ f_{n}} \partial_{t}u_{n}(t,\cdot)  \rvert_{0}=-\sqrt{ f_{1}(x_{a}(t))} \partial_{t}u_{1}(t,x_{a}(t)),\text{ }\forall t\in[0,t_{1}],\\
& x_{a}(0)=x_{1}.
\end{split}
\label{V2}
\end{equation} 
We seek now to obtain a similar relation for $\lvert \sqrt{ f_{1}} u_{1}(t,\cdot),..., \sqrt{ f_{n}} u_{n}(t,\cdot)  \rvert_{0}$.
One can show that it is possible to find $\psi_{0}$ that satisfies the previous hypothesis (\ref{nondegenerate0}) and (\ref{nondegenerate}) and such that in addition, 
there exists  $y_{2}\in[-1,1]$ such that
\begin{gather}
\left|\psi_{0}(y)\right|e^{-y}<\left|\psi_{0}(y_{2})\right|e^{-y_{2}},\text{ }\forall y\in[-1,1]\setminus \{y_{2}\},
\label{nondegenerate1}\\
\lvert \psi_{0}(y_{2})- \psi_{0}''(y_{2}) \rvert > 0,
\label{ym0}
\end{gather}
and such that there exists $m_{3}>0$ such that for all $m>m_{3}$, 
if  $\sup_{y\in[-1,1]}(\psi_{0}(y+y_{1})\frac{e^{-(y+y_{1})}}{\lambda_{i}(\frac{y}{m}+x_{0})})$ is achieved in $y_{m}\in[-1,1]$, then
 \begin{equation}
\lvert \psi_{0}(y_{m}+y_{1})- \psi_{0}''(y_{m}+y_{1}) \rvert > c_{1},
\label{ym}
 \end{equation}
where $c_{1}$ is a positive constant that does not depend on $m$. The example of $\psi_{0}$ provided in the Appendix is suitable. 
Thus with $h_{i}(l,y)=\frac{u^{0}_{i}}{\lambda_{i}(yl+x_{0})}\phi(y)e^{-y-y_{1}}$ one has:
\begin{equation}
\partial_{y} h_{i}(0,y_{2}-y_{1})=0.
\end{equation} 
Note that from (\ref{nondegenerate1}), $\psi_{0}(y_{2})=\psi_{0}'(y_{2})$, thus from (\ref{ym0})
\begin{equation}
\left|\partial_{yy} h_{i}(0,y_{2}-y_{1})\right|>0.
\end{equation} 
Therefore from the implicit function theorem, there exists $m_{4}>m_{3}$ such that for all $m>m_{4}$ 
and each $i\in[1,n]$ there exists a unique $y_{i}\in[-1-y_{1},1-y_{1}]$ such that
\begin{gather}
\partial_{y} h_{i}\left(\frac{1}{m},y_{i}\right)=0,\\
\left|y_{i}-(y_{2}-y_{1})\right|\leq\frac{C_{a}}{m},
\label{encadrementyi}
\end{gather}
where $C_{a}$ is a constant independent of $m$.
From (\ref{ineqi0})
there exists $m_{5}>m_{4}$
such that for all $m>m_{5}$,
\begin{equation}
\left| \frac{u_{i}^{0}}{\lambda_{i}\left(\frac{y_{i}}{m}+x_{0}\right)} \right| C_{b} <  \left|\frac{u_{i_{0}}^{0}}{\lambda_{i_{0}}\left(\frac{y_{i_{0}}}{m}+x_{0}\right)}\right|,\text{ }\forall\text{ }i\neq i_{0},
\end{equation} 
where $C_{b}>1$ is a constant independent of $m$. From (\ref{encadrementyi}), we have for any $i\in[1,n]$
\begin{equation}
\left|\frac{\phi(y_{i_{0}})e^{-y_{i_{0}}}}{\phi(y_{i})e^{-y_{i}}}\right|\geq 1-\frac{C_{r}}{m},
\end{equation} 
where $C_{r}$ is a constant independent of $m$. Therefore there exists $m_{6}>m_{5}$ such that for all $m>m_{6}$
\begin{equation}
\left| \frac{u_{i}^{0}}{\lambda_{i}\left(\frac{y_{i}}{m}+x_{0}\right)}\phi(y_{i})e^{-y_{i}} \right| \frac{(1+C_{b})}{2}< \left|\frac{u_{i_{0}}^{0}}{\lambda_{i_{0}}\left(\frac{y_{i_{0}}}{m}+x_{0}\right)}\phi(y_{i_{0}})e^{-y_{i_{0}}}\right|, \text{ }\forall\text{ }i\neq i_{0}.
\label{ineq221}
\end{equation} 
This means that for all $m>m_{6}$ there exists a unique $i_{0} \in[1,n]$ and a unique $x_{a_{0}}\in[x_{0}-\varepsilon,x_{0}+\varepsilon]$ such that
\begin{gather}
\lvert\sqrt{f_{i_{0}}(x_{a_{0}})} u_{i_{0}}(0,x_{a_{0}})\rvert=\sup\limits_{i\in[1,n],x\in[0,L]}\lvert \sqrt{f_{i}} u_{i}(0,\cdot)\rvert.
\label{encadrementz}
\end{gather} 
Now if we denote $g(t,x):=\partial_{x}(\sqrt{f_{i_{0}}(x)} u_{i_{0}}(t,x)\sgn(u_{i_{0}}(0,x_{a_{0}})))$, one has that
\begin{equation}
g(0,x_{a_{0}})=0,
\end{equation}
hence
\begin{equation}
\frac{-\lambda_{i_{0}}'(x_{a_{0}})}{m\lambda_{i_{0}}(x_{a_{0}})}\chi(x_{a_{0}})+\frac{\chi'(x_{a_{0}})}{m}=\chi(x_{a_{0}}).
\label{chiderivee}
\end{equation}
Therefore
\begin{equation}
\begin{split}
\partial_{x}g(0,x_{a_{0}})=&-\sgn(\lambda_{i_{0}})\frac{\lvert u_{i_{0}}^{0}\rvert}{m}e^{-m(x_{a_{0}}-x_{0})-y_{1}}\left(\left(\frac{1}{\lambda_{i_{0}}}\right)''(x_{a_{0}})\chi(x_{a_{0}})+\chi''(x_{a_{0}})\frac{1}{\lambda_{i_{0}}(x_{a_{0}})}\right.\\
&\left.+2\chi'(x_{a_{0}})\left(\frac{1}{\lambda_{i_{0}}}\right)'(x_{a_{0}})-m\left(\frac{\chi}{\lambda_{i_{0}}}\right)'(x_{a_{0}})\right.\\
&\left.-m\left(\left(\frac{1}{\lambda_{i_{0}}}\right)'(x_{a_{0}})\chi(x_{a_{0}})-m\frac{\chi(x_{a_{0}})}{\lambda_{i_{0}}(x_{a_{0}})}+\chi'(x_{a_{0}})\frac{1}{\lambda_{i_{0}}(x_{a_{0}})}\right)\right).
\end{split}
\end{equation}
Defining $c_{i0}:=-\sgn(\lambda_{i_{0}})\lvert u_{i_{0}}^{0}\rvert$ which is a non-zero constant, we have from (\ref{phi}) and the definition of $\phi$
\begin{equation}
\partial_{x}g(0,x_{a_{0}})=c_{i0}m\frac{e^{-y_{i_{0}}-y_{1}}}{\lambda_{i_{0}}\left(\frac{y_{i_{0}}}{m}+x_{0}\right)}\left(\psi''_{0}(y_{i_{0}}+y_{1})-2\psi'_{0}(y_{i_{0}}+y_{1})+\psi_{0}(y_{i_{0}}+y_{1})+O\left(\frac{1}{m^{2}}\right)+O\left(\frac{1}{m}\right)\right).
\end{equation}
Observe that, by definition, $y_{i_{0}}$ maximises $\left|\psi_{0}(y+y_{1})\frac{e^{-y-y_{1}}}{\lambda_{i_{0}}(\frac{y}{m}+x_{0})}\right|$, therefore we have from (\ref{ym}) and (\ref{chiderivee}) 
\begin{equation}
\begin{split}
\lvert \partial_{x}g(0,x_{a_{0}}) \rvert=&\lvert c_{a_{0}}\rvert m\left|\frac{e^{-y_{i_{0}}-y_{1}}}{\lambda_{i_{0}}\left(\frac{y_{i_{0}}}{m}+x_{0}\right)}\left(\psi''_{0}(y_{i_{0}}+y_{1})-\psi_{0}(y_{i_{0}}+y_{1})+O\left(\frac{1}{m}\right)\right)\right|,\\
=&\lvert c_{i_{0}}\rvert m\left|\frac{e^{-y_{i_{0}}-y_{1}}}{\lambda_{i_{0}}(\frac{y_{i_{0}}}{m}+x_{0})}\right|\left(c_{1}+O\left(\frac{1}{m}\right)\right).
\end{split}
\end{equation}
Hence, as the inequality (\ref{ineq221}) is strict and from the implicit function theorem, there exists $m_{7}>m_{6}$ such that for all $m>m_{7}$, $x_{b}\in C^{1}([0,t_{2}];[0,L])$ and $i_{0}\in[1,n]$ such that
\begin{equation}
\begin{split}
\lvert \sqrt{ f_{1}} u_{1}(t,\cdot),..., \sqrt{ f_{n}} u_{n}(t,\cdot)  \rvert_{0}=\sqrt{f_{i_{0}}(x_{b}(t))} u_{i_{0}}(t,x_{b}(t))\sgn(u_{i_{0}}(0,x_{a_{0}})),\text{ }\forall t\in[0,t_{2}],
\text{       }x_{b}(0)=x_{a_{0}}.
 \label{V1}
\end{split}
\end{equation} 
Hence $V$ is $C^{1}$ on $[0,t_{3})$ where $t_{5}=\min(t_{1},t_{2})>0$ and, denoting $s_{a_{0}}:=\sgn(u_{i_{0}}(0,x_{a_{0}}))$, we have from the definition of $V$, (\ref{V2}) and (\ref{V1}) 
\begin{equation}
\begin{split}
\frac{dV}{dt}(0)&=-\sqrt{ f_{1}(x_{1})}\partial_{tt}u_{1}(0,x_{1})- \frac{\partial}{\partial x}(\sqrt{ f_{1}}\partial_{t}u_{1}(0,\cdot))(x_{1})\frac{d x_{a}}{dt}(0)\\
&+s_{a_{0}}\left(\sqrt{ f_{i_{0}}(x_{a_{0}})} \partial_{t} u_{i_{0}}(t,x_{a_{0}})+\frac{\partial}{\partial x}\left(\sqrt{ f_{i_{0}}} u_{i_{0}}(0,\cdot)\right)(x_{a_{0}})\frac{d x_{b}}{dt}(0)\right).
\end{split}
\end{equation}
But now observe that for a fixed $m$, $x_{a_{0}}$ is an interior maximum thus 
\begin{equation}
\frac{d}{dx}(\sqrt{ f_{i_{0}}}u_{i_{0}}(0,\cdot))(x_{a_{0}})=0.
\label{interiormax}
\end{equation} 
Also as $\frac{d}{dx}(\sqrt{f_{1}}\partial_{t} u_{1}(0,\cdot))(x_{1})=0$, we have 
\begin{equation}
\frac{dV}{dt}(0)=-\sqrt{ f_{1}(x_{1})}\partial_{tt}^{2}u_{1}(0,x_{1})+s_{a_{0}}\sqrt{ f_{i_{0}}(x_{a})} \partial_{t} u_{i_{0}}(t,x_{a_{0}}).
\label{dVt}
\end{equation} 
Besides as $\phi$ has compact support in $[-1-y_{1},1-y_{1}]$, we have
\begin{equation}
\left|e^{m(x-x_{0})+y_{1}}\chi(x)\right|\leq e^{1}\lVert\chi\rVert_{\infty},
\label{uniformaly1}
\end{equation} 
and the right-hand side does not depend on $m$, thus
\begin{equation}
\lim_{m\rightarrow +\infty}\left|\frac{e^{m(x-x_{0})+y_{1}}}{m}\chi(x)\right|=0,
\label{decay1}
\end{equation}
uniformally on $[0,L]$ and therefore in particular for $x_{a_{0}}$ (even though $x_{a_{0}}$ might depend on $m$). We denote 
\begin{equation}
V_{2}:=-\sqrt{ f_{1}(x_{a}(t))}\partial_{t}u_{1}(t,x_{a}(t)).          
\end{equation} 
Using (\ref{y1}) and $\frac{d}{dx}(\sqrt{f_{1}}\partial_{t} u_{1}(0,\cdot))(x_{1})=0$,
we have
\begin{equation}
\begin{split}
\frac{dV_{2}}{dt}(0)&=-\sqrt{f_{1}(x_{1})}\partial_{tt}^{2}u_{1}(0,x_{1})\\
 &=-\sqrt{f_{1}(x_{1})}\partial_{t}(-\lambda_{1}\partial_{x} u_{1}(\cdot,x_{1})-\sum\limits_{j=1}^{n}M_{1j}u_{j}(\cdot,x_{1}))(0)\\
 &=-\sqrt{f_{1}(x_{1})}(-\lambda_{1}\partial_{x} (\partial_{t} u_{1}(0,x_{1}))-\sum\limits_{j=1}^{n}M_{1j}\partial_{t}u_{j}(0,x_{1}))\\
 &=-\sqrt{f_{1}(x_{1})}(\lambda_{1}\frac{(\sqrt{f_{1}})'}{\sqrt{f_{1}}}\partial_{t}u_{1}(0,x_{1})-\sum\limits_{j=1}^{n}M_{1j}\partial_{t}u_{j}(0,x_{1}))\\
 &=-\sqrt{f_{1}(x_{1})}(\frac{\lambda_{1}f_{1}'}{2f_{1}}\partial_{t}u_{1}(0,x_{1})-\sum\limits_{j=1}^{n}M_{1j}\partial_{t}u_{j}(0,x_{1})).
 \end{split}
\end{equation} 
And from (\ref{quantite}) and (\ref{estimation})
\begin{equation}
\begin{split}
\frac{dV_{2}}{dt}(0)=&u_{1}^{0}\left(\frac{\lambda_{1}f_{1}'}{2f_{1}}\left(1+O\left(\frac{1}{m}\right)\right)\right.\\
&\left.-\sum\limits_{j=1}^{n}M_{1j}(0,x_{1})\frac{u_{j}^{0}}{u_{1}^{0}}\frac{\sqrt{f_{1}(x_{1})}}{\sqrt{f_{j}(x_{1})}}\left(1+O\left(\frac{1}{m}\right)+\frac{\sqrt{f_{j}}(x_{j})\partial_{t}u_{j}(x_{j})-\sqrt{f_{j}}(x_{1})\partial_{t}u_{j}(x_{1})}{u_{j}^{0}}\right)\right).
\end{split}
\label{99}
\end{equation}
We know that if $M_{1j}(0,x_{0})\neq0$, then there exists $m_{8}\in\mathbb{N}^{*}$ such that for all $m>m_{8}$, $\sgn(M_{1j}(0,x_{0}))=\sgn(M_{1j}(0,x_{1}))$. We denote by $\mathcal{N}$ the
subset of $j\in\{1,...,n\}$ such that $M_{1j}(0,x_{0})=0$.  
Therefore from (\ref{99}) and (\ref{y01})
\begin{equation}
\begin{split}
\frac{dV_{2}}{dt}(0)=&u_{1}^{0}\left(\left[\frac{\lambda_{1}f_{1}'}{2f_{1}}-M_{11}(0,x_{1})+\sum\limits_{j=2, j\in \mathcal{N}^{c}}^{n}\left|M_{1j}(0,x_{1})\right|\left(1-\frac{1}{k}\right)\frac{\sqrt{f_{1}}}{\sqrt{f_{j}}}\right]\right.\\
&\left.+O\left(\frac{1}{m}\right)+\sum\limits_{j=0}^{n}C_{j}\left(\frac{\sqrt{f_{j}}(x_{j})\partial_{t}u_{j}(x_{j})-\sqrt{f_{j}}(x_{1})\partial_{t}u_{j}(x_{1})}{u_{j}^{0}}\right)\right),
\end{split}
\label{Cj}
\end{equation} 
where $C_{j}$ are constants that do not depend on $m$. 
Now, keeping in mind (\ref{dVt}), we are going to add $s_{a_{0}}\sqrt{ f_{i_{0}}(x_{a_{0}})} \partial_{t} u_{i_{0}}(0,x_{a_{0}})$ to obtain $dV/dt$ at $t=0$. 
But first observe that using (\ref{interiormax}) and
(\ref{uniformaly1})
\begin{equation}
\begin{split}
\sqrt{ f_{i_{0}}(x_{a_{0}})} \partial_{t} u_{i_{0}}(0,x_{a_{0}})&=\sqrt{ f_{i_{0}}(x_{a_{0}})} (-\lambda_{i} \partial_{x} u_{i_{0}}(0,x_{a_{0}})-\sum\limits_{j=1}^{n}M_{i_{0}j} u_{j}(0,x_{a_{0}}))\\
&=\sqrt{ f_{i_{0}}(x_{a_{0}})} \left(\lambda_{i} \frac{(\sqrt{ f_{i_{0}}})'(x_{a_{0}})}{\sqrt{ f_{i_{0}}(x_{a_{0}})}}\frac{u_{i_{0}}^{0}}{m}\chi(x_{a_{0}})\frac{e^{-m(x_{a_{0}}-x_{0})-y_{1}}}{\lambda_{i}\sqrt{f_{i_{0}}(x_{a_{0}})}}-\sum\limits_{j=1}^{n}M_{i_{0}j} \frac{u_{j}^{0}}{m}\chi(x_{a_{0}})\frac{e^{-m(x_{a_{0}}-x_{0})-y_{1}}}{\lambda_{i}\sqrt{ f_{i}(x_{a_{0}})}}\right)\\
&=O\left(\frac{1}{m}\right).
\end{split}
\end{equation} 
Therefore
\begin{equation}
\begin{split}
\frac{dV}{dt}(0)=&\frac{dV_{2}}{dt}(0)+O\left(\frac{1}{m}\right)\\
&=u_{1}^{0}\left(\left[\frac{\lambda_{1}f_{1}'}{2f_{1}}-M_{11}(0,x_{1})+\sum\limits_{j=2, j\in \mathcal{N}^{c}}^{n}\left|M_{1j}(0,x_{1})\right|\left(1-\frac{1}{k}\right)\frac{\sqrt{f_{1}}}{\sqrt{f_{j}}}\right]\right.\\
&\left.+O\left(\frac{1}{m}\right)+\sum\limits_{j=0}^{n}C_{j}\left(\frac{\sqrt{f_{j}}(x_{j})\partial_{t}u_{j}(x_{j})-\sqrt{f_{j}}(x_{1})\partial_{t}u_{j}(x_{1})}{u_{j}^{0}}\right)\right)
+O\left(\frac{1}{m}\right).
\end{split}
\label{expression1}
\end{equation} 
And from (\ref{quantite}) and the definition of $x_{j}$
\begin{equation}
\lim_{m\rightarrow+\infty}\left(\frac{\sqrt{f_{j}}(x_{j})\partial_{t}u_{j}(x_{j})-\sqrt{f_{j}}(x_{1})\partial_{t}u_{j}(x_{1})}{u_{j}^{0}}\right)=0.
\label{lim2}
\end{equation} 
Note that $x_{1}$ and $x_{j}$ both depend on $m$ and tend to $x_{0}$ when $m$ goes to infinity. Also we know that for all $m>m_{2}$, we have $x_{1}\in[x_{0}-\varepsilon,x_{0}+\varepsilon]$. Thus from (\ref{lemma1cont}),
\begin{gather}
\lim_{m\rightarrow+\infty}\left[\frac{\lambda_{1}(x_{1})f_{1}'(x_{1})}{2f_{1}(x_{1})}-M_{11}(0,x_{1})+\sum\limits_{j=2,j\in \mathcal{N}^{c}}^{n}\left|M_{1j}(0,x_{1})\right|\left(1-\frac{1}{k}\right)\frac{\sqrt{f_{1}(x_{1})}}{\sqrt{f_{j}(x_{1})}}\right]>0.
\label{ineq22}
\end{gather} 
Therefore there exists $m_{9}>0$ such that for all $m>m_{9}$
\begin{equation}
\frac{dV}{dt}(0)>0.
\end{equation} 
But we know \textcolor{black}{from (\ref{decroissance1})} that
\begin{equation}
\frac{dV}{dt}(0)\leq-\gamma V(0)<0.
\label{contra}
\end{equation} 
\textcolor{black}{Note that (\ref{contra}) is true 
as $V$ is $C^{1}$ in $[0,t_{1})$ and from \eqref{decroissance1}, for any $t\in[0,t_{1})$,
\begin{equation}
 \frac{V(t)-V(0)}{t}\leq V(0)\frac{e^{-\gamma t}-1}{t}
\end{equation} 
which, letting $t\rightarrow 0$, gives (\ref{contra}) and a contradiction.}
This ends the proof of Theorem \ref{resultat2}.
\end{proof}
\label{s4}

\section{Further details}
The previous results were derived for the $C^{1}$ norm but actually they can be extended to the $C^{p}$ norm, for $p\in\mathbb{N}^{*}$, with the same conditions. 
Namely we can extend the definition of \textit{basic $C^{p}$ Lyapunov function} for $p\in\mathbb{N}^{*}$ by replacing $V$ in Definition \ref{defC1} by
\begin{equation}
V(\mathbf{u}(t,\cdot))=\sum\limits_{k=0}^{p}\left| \textcolor{black}{\sqrt{f_{1}}} (\textcolor{black}{E\partial_{t}^{k}\mathbf{u}(t,\cdot)})_{1},..., \textcolor{black}{\sqrt{f_{n}}}(\textcolor{black}{E\partial_{t}^{k}\mathbf{u}(t,\cdot)})_{n}\right|_{0}.
\end{equation}
Defining the $p-1$ compatibility conditions  as in \cite{Coron1D} at (4.136) (see also (4.137)-(4.142)), the well-posedness still holds \cite{Coron1D} and we can state:
\begin{thm}
Let  a quasilinear hyperbolic system be of the form (\ref{system21}),(\ref{nonlocal}), with $A$ and $B$ of class $C^{p}$, $\Lambda$ defined as in (\ref{Lambda}) and $M$ as in (\ref{M2}), if 
\begin{enumerate}
\item (Interior condition) the system
\begin{equation}
\Lambda_{i}f_{i}'\leq -2\left(-M_{ii}(0,x)f_{i} + \sum\limits_{k=1,k\neq i}^{n}\lvert M_{ik}(0,x)\rvert \frac{f_{i}^{3/2}}{\sqrt{f_{k}}} \right),
\label{cond1111}
\end{equation} 
admits a solution $(f_{1},...,f_{n})$ on $[0,L]$ such that for all $i\in[1,n]$, $f_{i}>0$,
\item (Boundary condition) there exists a diagonal matrix $\Delta$ with positive coefficients such that
\begin{equation}
\lVert \Delta G'(0)\Delta^{-1}\rVert_{\infty}<\frac{\inf_{i}\left(\frac{f_{i}(d_{i})}{\Delta_{i}^{2}}\right)}{\sup_{i}\left(\frac{f_{i}(L-d_{i})}{\Delta_{i}^{2}}\right)},
\label{condauxbords1}
\end{equation} 
where 
$d_{i}=L$ if $\Lambda_{i}>0$, and $d_{i}=0$ otherwise.

\end{enumerate}
Then there exists a basic $C^{p}$ Lyapunov function for the system (\ref{system21}),(\ref{nonlocal}).
\label{resultat11}
\end{thm}
\begin{thm}
Let a quasilinear hyperbolic system be of the form (\ref{system21}) with $A$ 
and $B$ of class $C^{p+2}$,
 there exists a control of the form (\ref{nonlocal}) such that there exists a basic $C^{p}$ Lyapunov function if and only if
\begin{equation}
\Lambda_{i}f_{i}'\leq -2\left(-M_{ii}(0,x)f_{i} + \sum\limits_{k=1,k\neq i}^{n}\lvert M_{ik}(0,x)\rvert \frac{f_{i}^{3/2}}{\sqrt{f_{k}}} \right),
\label{cond1222}
\end{equation} 
admits a solution $(f_{1},...,f_{n})$ on $[0,L]$ such that for all $i\in[1,n]$, $f_{i}>0$.
\label{resultat22}
\end{thm}
A proof of this is included in the Appendix (see \ref{extension}).
\vspace{\baselineskip}\\
This article therefore fills the blank 
about the exponential stability for the $C^{p}$ norm for quasilinear hyperbolic systems with non-zero source term using a Lyapunov approach, for any $p\in\mathbb{N}^{*}$. 

We introduced the notion of \textit{basic $C^{1}$ Lyapunov function} that can be seen as natural Lyapunov function for the $C^{1}$ norm. 
For general quasilinear hyperbolic systems we gave a sufficient interior condition on the system and a sufficient boundary condition such that there exists a basic $C^{1}$ Lyapunov function that ensure exponential stability of the system for the $C^{1}$ norm. We also showed that the interior condition is necessary for the existence of such basic $C^{1}$ Lyapunov function. Therefore in some cases, there cannot exist such basic $C^{1}$ Lyapunov function whatever the 
boundary conditions are.

\section{Acknowledgement}
The author would like to thank his advisor Jean-Michel Coron for suggesting this problem, for his constant support, and for many fruitful discussions. The author would also like to thank Georges Bastin, Sébastien Boyaval, Nicole Goutal, Peipei Shang, Shengquan Xiang and Christophe Zhang for their valuable remarks, and many interesting discussions. The author would also like to thank the ETH - FIM and the ETH - ITS for their support and their warm welcome. Finally the author would like to thank the ANR project Finite 4SoS ANR 15-CE23-0007.

\appendix
\section{Appendix}

\subsection{Bound on the derivative of $W_{2,p}$}
\label{W2derivative}
\paragraph{Derivative of $W_{2,p}$}  

Recall that we have from (\ref{W2})
\begin{equation*}
 W_{2,p}=\left(\int_{0}^{L} \sum\limits_{i=1}^{n}   f_{i}(x)^{p} (\textcolor{black}{E\mathbf{u}_{t}})_{\textcolor{black}{i}}^{2p} e^{-2p\mu s_{i} x} dx \right)^{1/2p},
\end{equation*} 
\textcolor{black}{where $E=E(\mathbf{u}(t,x),x)$ given by \eqref{defm}--\eqref{condId}} and that $\mathbf{u}_{t}$ satisfies the following equation
\begin{equation}
 \mathbf{u}_{tt}+A(\mathbf{u},x)u_{tx}+\left[\frac{\partial A}{\partial  \mathbf{u}}(\mathbf{u},x).\mathbf{u}_{t}\right]\mathbf{u}_{x}+ \frac{\partial B}{\partial  \mathbf{u}}(\mathbf{u},x)\mathbf{u}_{t}=0,
\end{equation} 
where $\partial A/\partial \mathbf{u}.u_{t}$ is the matrix with coefficients $\sum\limits_{k=1}^{n}\partial A_{ij}/\partial \mathbf{u}_{k}(\mathbf{u},x).\partial_{t}\mathbf{u}_{k}(t,x)$.
We can again differentiate $W_{2,p}$ with respect to time along the trajectories which are of class $C^{2}$ (recall that we are proving the estimate (\ref{IWp}) for $C^{2}$ solutions first). Using integration by parts 
 as previously:
\begin{equation}
\begin{split}
 \frac{dW_{2,p}}{dt}= & -\frac{W_{2,p}^{1-2p}}{2p}\left[\sum\limits_{i=1}^{n}\lambda_{i} f_{i}(x)^{p} (\textcolor{black}{E}\mathbf{u}_{t})_{i}^{2p} e^{-2p\mu s_{i}x}\right]_{0}^{L} 
 \\  &-W_{2,p}^{1-2p}\int_{0}^{L}\sum\limits_{i=1}^{n}  f_{i}(x)^{p}(\textcolor{black}{E}\mathbf{u}_{t})_{i}^{2p-1} \left[\left(\textcolor{black}{E}\left(D_{a}+\frac{\partial B}{\partial \mathbf{u}}(\mathbf{u},x)\right).\mathbf{u}_{t}\right)_{\textcolor{black}{i}}\right.
 \\  &\left.\textcolor{black}{- \left(\left(\frac{\partial E}{\partial \mathbf{u}}.\mathbf{u}_{t}\right)\mathbf{u}_{t}+\lambda\left(\frac{\partial E}{\partial \mathbf{u}}.\mathbf{u}_{x}\right)\mathbf{u}_{t}+\lambda(\partial_{x} E)\mathbf{u}_{t}\right)_{i}} \right] e^{-2p\mu s_{i}x} dx 
 \\  &+\frac{W_{2,p}^{1-2p}}{2}\int_{0}^{L}\sum\limits_{i=1}^{n} \left(\lambda_{i}(\mathbf{u},x)  f_{i}(x)^{p-1}f'_{i}(x) (\textcolor{black}{E}\mathbf{u}_{t})_{i}^{2p}\right.
 \\  &\left.+\frac{d}{dx}\frac{(\lambda_{i}(\mathbf{u},x))}{p} f_{i}(x)^{p}(\textcolor{black}{E}\mathbf{u}_{t})_{i}^{2p}\right)e^{-2p\mu s_{i}x} dx
 \\  &-\mu W_{2,p}^{1-2p} \int_{0}^{L}\sum\limits_{i=1}^{n} \lvert\lambda_{i}\rvert   f_{i}^{p}(x)(\textcolor{black}{E}\mathbf{u}_{t})_{i}^{2p}e^{-2p\mu s_{i}x} dx,
 \end{split}
 \label{682}
\end{equation} 
where $D_{a}$ is the matrix with coefficient $\sum\limits_{k=1}^{n}(\partial A_{ik}/\partial u_{j})(\mathbf{u}_{x})_{k}$, so that $D_{a}.\mathbf{u}_{t}=\left[\frac{\partial A}{\partial  \mathbf{u}}(\mathbf{u},x).\mathbf{u}_{t}\right]\mathbf{u}_{x}$.
\textcolor{black}{
Observe that $E$ is $C^{2}$ and invertible by definition (given by \eqref{defm}--\eqref{condId}),
thus $\mathbf{u}_{t}=E^{-1} (E\mathbf{u}_{t})$. We can therefore}
denote, similarly as previously
\begin{equation}
\mathrm{I}_{21}:=\frac{W_{2,p}^{1-2p}}{2p}\left[\sum\limits_{i=1}^{n}\lambda_{i}   f_{i}(x)^{p} (\textcolor{black}{E}\mathbf{u}_{t})_{i}^{2p} e^{-2p\mu s_{i} x}\right]_{0}^{L},
\label{I21}
\end{equation} 
and 
\begin{equation}
\begin{split}
\mathrm{I}_{31}=& W_{2,p}^{1-2p}\left(\int_{0}^{L} \sum\limits_{i=1}^{n}   f_{i}(x)^{p} (\textcolor{black}{E}\mathbf{u}_{t})_{i}^{2p-1} \left(\sum\limits_{k=1}^{n}R_{ik}(\mathbf{u},x)(E\mathbf{u}_{t})_{k}\right) e^{-2p\mu s_{i}x} dx \right)
\\ & -\frac{W_{2,p}^{1-2p}}{2}\int_{0}^{L}\sum\limits_{i=1}^{n} \lambda_{i}(\mathbf{u},x)   f_{i}(x)^{p-1}f'_{i}(x) (\textcolor{black}{E}\mathbf{u}_{t})_{i}^{2p}e^{-2p\mu s_{i} x} dx.
\label{I31}
\end{split}
\end{equation} 
\textcolor{black}{where $R=(R_{ij})_{(i,j)\in[1,n]^{2}}$ is defined as $R:=E\left(D_{a}+\frac{\partial B}{\partial \mathbf{u}}\right)E^{-1}$. 
As $E$ is $C^{1}$ and its inverse is continuous, and from \eqref{condId}, there exists a constant $C_{0}$ independant of $\mathbf{u}$ (and $p$) such that
\begin{equation}
\max\limits_{(i,j)\in[1,n]^{2}}\left|\left(\left(\frac{\partial E}{\partial \mathbf{u}}.\mathbf{u}_{t}\right)E^{-1}+\left(\frac{\partial E}{\partial \mathbf{u}}.\mathbf{u}_{x}\right)E^{-1}+(\partial_{x} E)E^{-1}\right)_{ij}\right|
\leq C_{0}|\mathbf{u}|_{1}.
\end{equation} 
Note that we used \eqref{condId} and the fact that $\partial_{x}(E(\mathbf{0},x))=0$. Thus,}
similarly as for (\ref{decomp3}), we have
\begin{equation}
\frac{dW_{2,p}}{dt}\leq-I_{21}-I_{31}-(\mu \alpha_{0}- \frac{C_{6}}{2p}) W_{2,p} +C_{7}W_{2,p}\lvert\mathbf{u}\rvert_{1},
\label{decomp22}
\end{equation} 
where $C_{6}$ and $C_{7}$ are constants that does not depend on $p$ or $\mathbf{u}$ provided that $\lvert\mathbf{u}\rvert_{1}<\eta$ for $\eta$ small enough but independent of $p$. \textcolor{black}{Recall that $\alpha_{0}$ is defined in Section \ref{s4} right before \eqref{decomp1}}. Just as previously, a sufficient condition such that there exist $p_{1}\in\mathbb{N}^{*}$, $\eta_{1}>0$ and $\mu_{1}$ such that $I_{31}>0$ for $\mu<\mu_{1}$, $p>p_{1}$ and $\lvert \mathbf{u} \rvert_{1}<\eta_{1}$ is 
\begin{equation}
-\lambda_{i} \frac{f'_{i}}{f_{i}}>2\sum\limits_{k=1,k\neq i}^{n}\lvert R_{ik}(\mathbf{u},x)\rvert\sqrt{\frac{f_{i}}{f_{k}}} - 2 R_{ii},
\end{equation} 
But we have from the definition of $D_{a}$\textcolor{black}{, \eqref{condId} and \eqref{M2}}:
\begin{equation}
\textcolor{black}{E}\left(D_{a}+\frac{\partial B}{\partial \mathbf{u}}\right)\textcolor{black}{E^{-1}}=\frac{\partial B}{\partial \mathbf{u}}(0,x)+O(\lvert \mathbf{u} \rvert_{1})\text{ }\textcolor{black}{=M(0,x)+O(\lvert \mathbf{u} \rvert_{1})},
\label{da1}
\end{equation}
and recall that in the proof $(f_{1},...,f_{n})$ have been selected such that
\begin{equation}
-\Lambda_{i} \frac{f'_{i}}{f_{i}}>2\sum\limits_{k=1,k\neq i}^{n}\lvert M_{ik}(0,x)\rvert\sqrt{\frac{f_{i}}{f_{k}}} - 2 M_{ii}(0,x).
\label{da2}
\end{equation} 
Thus from (\ref{da1}) and (\ref{da2}) there exist $\eta_{2}>0$, $p_{1}\in\mathbb{N}^{*}$ and $\mu_{1}$ such that if $\mu<\mu_{1}$, $p>p_{1}$ and $\lvert \mathbf{u} \rvert_{1}<\eta_{2}$, 
then $I_{31}>0$. 
It remains to deal with $I_{21}$. \textcolor{black}{As $E$ is $C^{1}$, and from \eqref{condId},
\begin{equation}
(E\mathbf{u}_{t})=\mathbf{u}_{t}+(\mathbf{u}.\textcolor{black}{\mathcal{V}})\mathbf{u}_{t}
\label{daV}
\end{equation} 
where $\textcolor{black}{\mathcal{V}}=\textcolor{black}{\mathcal{V}}(\mathbf{u}(t,x),x)$ is continuous on $\mathcal{B}_{\eta_{0}}\times[0,L]$.
Using \eqref{daV} together with} (\ref{I21}) and proceeding exactly as previously for $I_{2}$, we get
\begin{equation}
\begin{split}
\mathrm{I}_{21}=&\frac{W_{2,p}^{1-2p}}{2p}\left(\sum\limits_{i=1}^{m}\lambda_{i}(\mathbf{u}(t,L),L)   f_{i}(L)^{p} ((\mathbf{u}_{t})_{i}(t,L)\textcolor{black}{+((\mathbf{u}(t,L).\textcolor{black}{\mathcal{V}})\mathbf{u}_{t}(t,L))_{i}})^{2p} e^{-2p\mu L}\right.
\\&\left.-\sum\limits_{i=1}^{m}\lambda_{i}(\mathbf{u}(t,0),0)  f_{i}(0)^{p} ((\mathbf{u}_{t})_{i}(t,0)\textcolor{black}{+((\mathbf{u}(t,0).\textcolor{black}{\mathcal{V}})\mathbf{u}_{t}(t,0))_{i}})^{2p} \right.\\
&\left. -\sum\limits_{i=m+1}^{n}\lvert\lambda_{i}(\mathbf{u}(t,L),L)\rvert   f_{i}(L)^{p} ((\mathbf{u}_{t})_{i}(t,L)\textcolor{black}{+((\mathbf{u}(t,L).\textcolor{black}{\mathcal{V}})\mathbf{u}_{t}(t,L))_{i}})^{2p} e^{2p\mu L}\right.\\
&\left.+\sum\limits_{i=m+1}^{n}\lvert\lambda_{i}(\mathbf{u}(t,0),0)\rvert   f_{i}(0)^{p} ((\mathbf{u}_{t})_{i}(t,0)\textcolor{black}{+((\mathbf{u}(t,0).\textcolor{black}{\mathcal{V}})\mathbf{u}_{t}(t,0))_{i}})^{2p}\right).
\end{split}
\end{equation} 
Recall that $K=G'(0)$ and $\Delta=\left(\Delta_{1},...,\Delta_{n}\right)^{T}\in(\mathbb{R}_{+}^{*})^{n}$ are chosen such that 
\begin{equation}
\theta:=\sup_{\lVert \xi \rVert_{\infty}\leq 1}(\sup_{i}(\lvert \sum\limits_{j=1}^{n}(\Delta_{i} K_{ij}\Delta_{j}^{-1})\xi_{j} \rvert))<\frac{\inf_{i}\left(\frac{f_{i}(d_{i})}{\Delta_{i}^{2}}\right)}{\sup_{i}\left(\frac{f_{i}(L-d_{i})}{\Delta_{i}^{2}}\right)}.
\end{equation} 
We denote again
\begin{gather}
 \xi_{i}:=\Delta_{i}(\mathbf{u}_{t})_{i}(t,L)\text{ for }i\in[1,m],\\
 \xi_{i}:=\Delta_{i}(\mathbf{u}_{t})_{i}(t,0)\text{ for }i\in[m+1,n].
\end{gather} 
From the fact that $G$ and $\mathbf{u}$ are $C^1$, we can differentiate (\ref{nonlocal}) with respect to time, and we have
\begin{equation}
\begin{pmatrix}
(\mathbf{u}_{t})_{+}(t,0)\\
(\mathbf{u}_{t})_{-}(t,L)
\end{pmatrix}
= K\begin{pmatrix}
(\mathbf{u}_{t})_{+}(t,L)\\
(\mathbf{u}_{t})_{-}(t,0)
\end{pmatrix}
+o\left(\left| \begin{pmatrix}
(\mathbf{u}_{t})_{+}(t,L)\\
(\mathbf{u}_{t})_{-}(t,0)
\end{pmatrix}
\right| 
\right),
\end{equation} 
where $o(x)$ refers to a function such that $o(x)/\lvert x \rvert$ tends to $0$ when $\lvert\mathbf{u}\rvert_{\textcolor{black}{1}}$ tends to $0$. Thus
\begin{equation}
\begin{split}
\mathrm{I}_{21}=&\frac{W_{2,p}^{1-2p}}{2p}\left(\sum\limits_{i=1}^{m} \lambda_{i}(\mathbf{u}(t,L),L) \frac{f_{i}(L)^{p}}{\Delta_{i}^{2p}} ((\mathbf{u}_{t})_{i}(t,L)\Delta_{i}\textcolor{black}{+o(|\xi|)})^{2p} e^{-2p\mu L}\right.
\\ &\left. +\sum\limits_{i=m+1}^{n} \lvert \lambda_{i}(\mathbf{u}(t,0),0)\rvert \frac{f_{i}(0)^{p}}{\Delta_{i}^{2p}}((\mathbf{u}_{t})_{i}(t,0)\Delta_{i}\textcolor{black}{+o(|\xi|)})^{2p}\right.
\\ &\left.-\sum\limits_{i=1}^{m}\lambda_{i}(\mathbf{u}(t,0),0) \frac{f_{i}(0)^{p}}{\Delta_{i}^{2p}} (\sum\limits_{k=1}^{n}K_{ik}\xi_{k}(t)\frac{\Delta_{i}}{\Delta_{k}}\textcolor{black}{+o(|\xi|)})^{2p}\right.
\\ &\left.-\sum\limits_{i=m+1}^{n}\lvert\lambda_{i}(\mathbf{u}(t,L),L)\rvert \frac{f_{i}(L)^{p}}{\Delta_{i}^{2p}} (\sum\limits_{k=1}^{n}K_{ik}\xi_{k}(t)\frac{\Delta_{i}}{\Delta_{k}}\textcolor{black}{+o(|\xi|)})^{2p}e^{2p\mu L}\right)
\end{split}
\end{equation}
We end by proceeding exactly as for $I_{2}$. Therefore under assumption (\ref{condauxbords}), there exist $p_{3}$, $\mu_{3}$ and $\eta_{3}>0$ such that for $\mu<\mu_{3}$ and $\lvert \mathbf{u} \rvert_{1}<\eta_{3}$, $I_{21}<0$. 
Therefore, as stated in the main text, there exist $\eta_{4}$, $p_{5}$ and $\mu$ such that for all $p>p_{5}$ and $\lvert \mathbf{u} \rvert_{1}<\eta_{4}$
\begin{equation}
\frac{dW_{2,p}}{dt}\leq-\frac{\mu \alpha_{0}}{2} W_{2,p} +C_{7}W_{2,p}\lvert\mathbf{u}\rvert_{1}.
\end{equation}
\subsection{Existence of $\psi_{0}$}
\label{psi0existence}
We want to find a function $\psi_{0}$ that is $C^{1}$ with compact support in $[-1,1]$ such that there exists a unique $y_{1}\in(-1,1)$ such that
\begin{equation}
\lvert\psi_{1}(y)-\psi_{1}'(y)\rvert e^{-y}<\lvert\psi_{1}(y_{1})-\psi_{1}'(y_{1})\rvert e^{-y_{1}},\text{ }\forall y\in[-1,1]\setminus\{y_{1}\}.
\end{equation}
Let $\chi$ be a positive $C_{c}^{1}$ with compact support in in $[-1,1]$ such that 
\begin{equation}
\begin{split}
&\chi\equiv 1\text{ on }\left[-\frac{1}{2},\frac{1}{2}\right],\\
&\lvert\chi\rvert\leq 1\text{ on }\left[-1,1\right],\\
&\lvert\chi'\rvert\leq 3\text{ on }\left[-1,-\frac{1}{2}\right)\cup\left(\frac{1}{2},1\right],
\end{split}
\label{choixchi}
\end{equation}
and let us define $f:y\rightarrow e^{-n_{1}y^{2}}$ where $n_{1}\in\mathbb{N}^{*}$ will be chosen later on. We have

\begin{equation}
(f(y)-f'(y))e^{-y}=e^{-n_{1}y^{2}-y}(1+2n_{1}y).
\end{equation}
Therefore
\begin{equation}
\lvert f(y)-f'(y)\rvert e^{-y}\leq e^{-\frac{n_{1}}{4}+1}(1+2n_{1})\text{ on }\left[-1,-\frac{1}{2}\right)\cup\left(\frac{1}{2},1\right].
\label{ff'}
\end{equation} 
As  $\lim_{n\rightarrow+\infty} e^{-\frac{n}{4}+1}(1+2n)=0$ we can choose $n_{1}\geq1$ large enough such that
\begin{equation}
e^{-\frac{n_{1}}{4}+1}(1+2n_{1})\leq \frac{1}{3}.
\label{choixn}
\end{equation} 
Now let us consider $\psi_{1}=\chi f$, one \textcolor{black}{has}
\begin{equation}
\lvert\psi_{1}(y)-\psi_{1}'(y)\rvert e^{-y}=\lvert \chi(y) (f(y)-f'(y)) e^{-y}-\chi'(y)f(y)e^{-y} \rvert.
\end{equation}
Therefore from (\ref{choixchi}), (\ref{ff'}) and (\ref{choixn}), we have $\text{ on }\left[-1,-\frac{1}{2}\right)\cup\left(\frac{1}{2},1\right]$
\begin{equation}
\lvert\psi_{1}(y)-\psi_{1}'(y)\rvert e^{-y}\leq \frac{1}{3}+\frac{3}{9}<1.
\label{pp1}
\end{equation}
As $g:y\rightarrow \lvert\psi_{1}(y)-\psi_{1}'(y)\rvert e^{-y}$ has compact support on $[-1,1]$ we can define $d$ as
\begin{equation}
d:=\sup_{y\in[-1,1]}(\lvert\psi_{1}(y)-\psi_{1}'(y)\rvert e^{-y}),
\label{defad}
\end{equation}
and $d$ is attained in at least one point. But as $g(0)=1$ and as from (\ref{pp1}) $\lvert g \rvert<1$ on $\left[-1,-\frac{1}{2}\right)\cup\left(\frac{1}{2},1\right]$, $d$ is attained only on $\left(-\frac{1}{2},\frac{1}{2}\right)$, and on $\left(-\frac{1}{2},\frac{1}{2}\right)$ we have
\begin{equation}
\lvert\psi_{1}(y)-\psi_{1}'(y)\rvert e^{-y}=\lvert f(y)-f'(y)\rvert e^{-y}.
\end{equation}
Let us show now that $\lvert f(y)-f'(y)\rvert e^{-y}$ admits a unique maximum on $\left(-\frac{1}{2},\frac{1}{2}\right)$. We know that $\lvert f(y)-f'(y)\rvert e^{-y}$ attains a maximum $d\geq1$ on $\left(-\frac{1}{2},\frac{1}{2}\right)$ and when it attains this maximum $((f(y)-f'(y))e^{-y})'$ vanishes, therefore
\begin{equation}
e^{-n_{1}y^{2}-y}(2n_{1}-4n_{1}^{2}y^{2}-1-4n_{1}y)=0,
\end{equation}
hence
\begin{equation}
4n_{1}^{2}y^{2}+4n_{1}y+(1-2n_{1})=0.
\end{equation}
This equation has only two solutions: $y_{\pm}=\frac{-1\pm\sqrt{2n_{1}}}{2n_{1}}$ but
\begin{equation}
\lvert f(y_{-})-f'(y_{-})\rvert e^{-y_{-}}=\sqrt{2n_{1}}e^{\frac{1-2n_{1}+4\sqrt{2n_{1}}}{4n_{1}}}>\sqrt{2n_{1}}e^{\frac{1-2n_{1}-4\sqrt{2n_{1}}}{4n_{1}}}=\lvert f(y_{+})-f'(y_{+})\rvert e^{-y_{+}}.
\label{y-}
\end{equation}
Therefore $\lvert f(y)-f'(y)\rvert e^{-y}$ admits its maximum on $\left(-\frac{1}{2},\frac{1}{2}\right)$ at most one time. But we also know that it does admit a maximum on $\left(-\frac{1}{2},\frac{1}{2}\right)$, hence and from (\ref{pp1})
\begin{equation}
\exists ! y_{1}\in(-1,1):\text{ }\lvert\psi_{1}(y)-\psi_{1}'(y)\rvert e^{-y}<\lvert\psi_{1}(y_{1})-\psi_{1}'(y_{1})\rvert e^{-y_{1}},\text{ }\forall y\in[-1,1]\setminus\{y_{1}\}.
\end{equation}
Now we just need to normalize the function and define $\psi_{0}:=\frac{1}{d}\psi_{1}$ where $d$ is given in (\ref{defad}) to obtain the desired function $\psi_{0}$.\\
\vspace{\baselineskip}\\
Observe that this function also satisfies (\ref{nondegenerate1}), (\ref{ym0}) and (\ref{ym}):
Let $y\in\left(-1/2,1/2\right)$, then $\psi_{0}(y)e^{-y}$ is positive and one has
\begin{equation}
(\psi_{0}(\cdot)e^{-Id})'(y)= \frac{1}{d}(-1-2n_{1}y)e^{-y-n_{1}y^{2}},
\end{equation} 
thus on $\left(-1/2,1/2\right)$, $\lvert\psi_{0}\rvert e^{-Id}$ has a unique maximum achieved in $y_{2}=-1/2n_{1}$.
Now let $y\in[-1,1]\setminus\left(-1/2,1/2\right)$, we have from (\ref{choixn})
\begin{equation}
\lvert \psi_{0}(y)\rvert e^{-y} \leq \frac{e^{1-\frac{n_{1}}{4}}}{d} \leq \frac{e^{-\frac{3}{4n_{1}}}}{d}=\psi_{0}(y_{2})e^{-y_{2}}.
\end{equation}
Hence the function admit a unique maximum on $[-1,1]$ and (\ref{nondegenerate1}) is verified. And from (\ref{choixn}) we have
\begin{equation}
\psi_{0}(y_{2})-\psi_{0}''(y_{2})=(-2n_{1}+2)>0.
\end{equation} 
This implies (\ref{ym0}) and we are left with proving (\ref{ym}).
Let again $y+y_{1}\in[-1,1]\setminus\left(-1/2,1/2\right)$, for $i\in[1,n]$ 
and $m$ large enough, from (\ref{choixn})
\begin{equation}
\left|\psi_{0}(y+y_{1})\frac{e^{-y-y_{1}}}{\lambda_{i}(\frac{y}{m}+x_{0})}\right| \leq e^{-\frac{n_{1}}{4}}\frac{e^{y_{1}-y_{1}}}{d\inf_{[x_{0}-\frac{1+y_{1}}{m},x_{0}+\frac{(1-y_{1})}{m}]}\lvert\lambda_{1}\rvert}<\left|\psi_{0}(0)\frac{e^{y_{1}-y_{1}}}{\lambda_{1}(x_{0})}\right|,
\end{equation}
which means that $\sup\limits_{[-1,1]}\left|\psi_{0}(y+y_{1})\frac{e^{-y-y_{1}}}{\lambda_{i}(\frac{y}{m}+x_{0})}\right|$ can only be achieved on $(-1/2-y_{1},1/2-y_{1})$. But we also know that on $[-1/2,1/2]$, $\psi_{0}=d^{-1} f$. Therefore let be a $y_{m}$ maximizing $\sup_{[-1,1]}\left|\psi_{0}(y+y_{1})\frac{e^{-y-y_{1}}}{\lambda_{i}(\frac{y}{m}+x_{0})}\right|$, we know that $y_{m}$ exists as [-1,1] is a compact, that $y_{m}$ is an interior maximum and we have
\begin{equation}
\partial_{y}(e^{-n_{1}(y+y_{1})^{2}}\frac{e^{-y-y_{1}}}{d\lambda_{i}(\frac{y}{m}+x_{0})})(y_{m})=0.
\end{equation}
Hence
\begin{equation}
2 n_{1}\left((y_{m}+y_{1})+1+\frac{\lambda_{i}'(\frac{y_{m}}{m}+x_{0})}{m\lambda_{i}(\frac{y_{m}}{m}+x_{0})}\right)\frac{e^{-n_{1}^{2}(y_{m}+y_{1})-y_{m}-y_{1}}}{d\lambda_{i}(\frac{y_{m}}{m}+x_{0})}=0,
\end{equation}
thus
\begin{equation}
(y_{m}+y_{1})=-\frac{1}{2 n_{1}}-\frac{\lambda_{i}'(\frac{y_{m}}{m}+x_{0})}{(2 n_{1})m\lambda_{i}(\frac{y_{m}}{m}+x_{0})}.
\end{equation}
All it remains to show is that for $m$ large enough we have (\ref{ym}).  Let us compute $\psi_{0}''(y_{m}+y_{1})$
\begin{equation}
\begin{split}
\psi_{0}''(y_{m}+y_{1})&=d^{-1}f''(y_{m}+y_{1})=f(y_{m}+y_{1})(-2n_{1}+4n_{1}^{2}(y_{m}+y_{1})^{2})
\\ &=f(y_{m}+y_{1})(-2n_{1}+1+ \left(\frac{\lambda_{i}'(\frac{y_{m}}{m}+x_{0})}{m\lambda_{i}(\frac{y_{m}}{m}+x_{0})}\right)^{2}+\left(\frac{\lambda_{i}'(\frac{y_{m}}{m}+x_{0})}{m\lambda_{i}(\frac{y_{m}}{m}+x_{0})}\right)).
\end{split}
\end{equation}
Therefore there exists $m_{3}>0$ such that for all $m>m_{3}$,
\begin{equation}
\lvert\psi_{0}''(y_{m}+y_{1})-\psi_{0}(y_{m}+y_{1})\rvert>e^{-\frac{1}{n_{1}}}(2 n_{1}-3),
\end{equation}
and as we chose $n_{1}$ large enough, $C:=e^{-\frac{1}{n_{1}}}(2 n_{1}-3)>0$. This ends the proof of the existence of $\psi_{0}$.
\subsection{Adapting proof of Theorem \ref{resultat2} in the nonlinear case}
\label{adapting}
For all $\mathbf{u}(0,\cdot)\in\mathcal{B}_{\eta_{1}}$, 
we can still define
\begin{equation}
u_{i}(0,x)=\frac{u_{i}^{0}}{m}\chi(x)\frac{e^{-m(x-x_{0})-y_{1}}}{\Lambda_{i}(x)\sqrt{ f_{i}(x)}},
\label{y11}
\end{equation}
which is the analogous of (\ref{y1}) in the proof of Theorem \ref{resultat2}. If there are two index $i_{0}$ and $i_{1}$ such that $\min\limits_{i}(\lvert\Lambda_{i}(x_{0})\rvert)$ is achieved  we can still redefine 
$u_{i_{1}}^{0}$ as in (\ref{redefine}).
Observe then that if (\ref{lemma1cont0}) is satisfied, then there exists $\eta_{3}>0$ such that if $\lvert\mathbf{u}\rvert_{0}<\eta_{3}$ then
\begin{equation}
-\lambda_{i_{0}}(\mathbf{u},x_{0})f'_{i_{0}}(x_{0})<2\sum\limits_{k=1,k\neq i_{0}}^{n}\lvert M_{i_{0}k}(\mathbf{u},x_{0})\rvert\frac{f_{i_{0}}^{3/2}(x_{0})}{\sqrt{f_{k}(x_{0})}} - 2 M_{i_{0}i_{0}}(\mathbf{u},x_{0})f_{i_{0}}(x_{0}).
\end{equation}
From (\ref{y11}), (\ref{quantite}) becomes
\begin{equation}
\begin{split}
\sqrt{f_{i}}\partial_{t} u_{i}(0,x)=& -u_{i}^{0}e^{-m(x-x_{0})-y_{1}}\left[-
\chi(x)+\frac{\chi'(x)}{m} +\frac{\chi(x)\lambda_{i}\sqrt{f_{i}}}{m}\left(\frac{1}{\lambda_{i}\sqrt{f_{i}}}\right)'\right.\\ 
&\left.+\frac{1}{m}\sum\limits_{j=1}^{n} M_{ij}(\mathbf{u},x)\left(\frac{u_{j}^{0}}{u_{i}^{0}}\right)\left(\sqrt{\frac{f_{i}}{f_{j}}}\right)\frac{1}{\lambda_{j}}\chi(x)\right.\\
&\left.+\sum\limits_{j=1}^{n}\left(V_{ij}(\mathbf{u},x).\mathbf{u}(0,x)\right)\left(\frac{u_{j}^{0}}{u_{i}^{0}}\right)\left(\sqrt{\frac{f_{i}}{f_{j}}}\right)\frac{1}{\lambda_{j}}\left(-\chi(x)+\frac{\chi'(x)}{m} +\frac{\chi(x)\lambda_{j}\sqrt{f_{j}}}{m}\left(\frac{1}{\lambda_{j}\sqrt{f_{j}}}\right)' \right)
\right].
\end{split}
\end{equation} 
where $V_{ij}$ are $C^{2}$ functions as we assume
that $A$ is of class $C^{3}$. Therefore
\begin{equation}
\begin{split}
\sqrt{f_{i}}\partial_{t} u_{i}(0,x)=& -u_{i}^{0}e^{-m(x-x_{0})-y_{1}}\left[\left(-
\chi(x)+\frac{\chi'(x)}{m}\right)\left(1+\sum\limits_{j=1}^{n}\left(V_{ij}(\mathbf{u},x).\mathbf{u}(0,x)\right)\left(\frac{u_{j}^{0}}{u_{i}^{0}}\right)\left(\sqrt{\frac{f_{i}}{f_{j}}}\right)\frac{1}{\lambda_{j}}\right)\right.\\
&\left.+\chi(x)\left(\frac{\lambda_{i}\sqrt{f_{i}}}{m}\left(\frac{1}{\lambda_{i}\sqrt{f_{i}}}\right)'+\sum\limits_{j=1}^{n}\left(V_{ij}(\mathbf{u},x).\mathbf{u}(0,x)\right)\left(\frac{u_{j}^{0}}{u_{i}^{0}}\right)\left(\sqrt{\frac{f_{i}}{f_{j}}}\right)\frac{\sqrt{f_{j}}}{m}\left(\frac{1}{\lambda_{j}\sqrt{f_{j}}}\right)'\right) \right.\\ 
&\left.+\frac{1}{m}\sum\limits_{j=1}^{n} M_{ij}(\mathbf{u},x)\left(\frac{u_{j}^{0}}{u_{i}^{0}}\right)\left(\sqrt{\frac{f_{i}}{f_{j}}}\right)\frac{1}{\lambda_{j}}\chi(x)\right].
\end{split}
\label{quantitenonlinear}
\end{equation} 
Now, after the change of variable (\ref{phi}), one \textcolor{black}{has}
\begin{equation}
\begin{split}
\sqrt{f_{i}}\partial_{t} u_{i}(0,x)=& -u_{i}^{0}e^{-y-y_{1}}\left[-
\phi(y)\left(1+\frac{g_{i}(\mathbf{u}(0,y/m+x_{0}),y/m+x_{0})}{m}\right)\right.\\
&\left.+\phi'(y)\left(1+\frac{h_{i}(\mathbf{u}(0,y/m+x_{0}),y/m+x_{0})}{m}\right)\right].
\end{split}
\label{quantitenonlinear2}
\end{equation} 
where $h_{i}$ and $g_{i}$ are bounded functions in the $C^{2}$ norm and are independent of $m$. \textcolor{black}{Thus
\begin{equation}
\begin{split}
\sqrt{f_{i}}&(E\partial_{t} \mathbf{u})_{i}(0,x)= -\left(u_{i}^{0}+\sum\limits_{j=1}^{n}\sqrt{\frac{f_{i}}{f_{j}}}Z_{ij}(\mathbf{u}(0,y/m+x_{0}),y/m+x_{0})u_{j}^{0}\right) e^{-y-y_{1}}\left[-
\phi(y)+\phi'(y)\right]\\
&-u_{i}^{0}e^{-y-y_{1}}\left[-\phi(y)\left(\frac{g_{i}(\mathbf{u}(0,y/m+x_{0}),y/m+x_{0})}{m}\right)+\phi'(y)\left(\frac{h_{i}(\mathbf{u}(0,y/m+x_{0}),y/m+x_{0})}{m}\right)\right]\\
&-\sum\limits_{j=1}^{n}\sqrt{\frac{f_{i}}{f_{j}}}Z_{ij}(\mathbf{u}(0,y/m+x_{0}),y/m+x_{0})u_{j}^{0}e^{-y-y_{1}}\left[-\phi(y)\left(\frac{g_{j}(\mathbf{u}(0,y/m+x_{0}),y/m+x_{0})}{m}\right)\right.\\
&\left.+\phi'(y)\left(\frac{h_{j}(\mathbf{u}(0,y/m+x_{0}),y/m+x_{0})}{m}\right)\right].
\end{split}
\label{withZ}
\end{equation} 
where $Z(\mathbf{u},x)=(Z_{ij})_{(i,j)\in[1,n]^{2}}:=\mathbf{u}.V(\mathbf{u},x)$, with $V$ given by \eqref{daV}.}
In addition one also \textcolor{black}{has}
\begin{equation}
u_{i}(0,\frac{y}{m}+x_{0})=\frac{u_{i}^{0}}{m}\phi(y)\frac{e^{-y-y_{1}}}{\Lambda_{i}(\frac{y}{m}+x_{0})\sqrt{f_{i}(\frac{y}{m}+x_{0})}},\text{ }\forall i\in[1,n]
\label{unl}.
\end{equation} 
Thus, the function $y\rightarrow \mathbf{u}(0,y/m+x_{0})$ is $O(1/m)$ in the $C^{2}$ norm, which means that $g_{i}(\mathbf{u}(0,y/m+x_{0}),y/m+x_{0})$ and $g_{i}(\mathbf{u}(0,y/m+x_{0}),y/m+x_{0})$ are $O(1)$ in the $C^{2}$ norm when $m$ tends to $+\infty$.
\textcolor{black}{
Similarly, $Z$ is a $C^{2}$ function as $E$ is a $C^{3}$ function (recall that $A$ is $C^{3}$ for Theorem \ref{resultat2}), and there exists a constant $C$ independant of $\mathbf{u}$ and $m$ such that $\max_{(i,j)\in[1,n]^{2}}\left|Z_{ij}(\mathbf{u}(0,y/m+x_{0}),(0,y/m+x_{0}))\right|\leq C|\mathbf{u}(0,y/m+x_{0})|$. 
This, with \eqref{unl}, implies that the terms which involves $Z$ in \eqref{withZ} are all $O\left(1/m\right)$ in the $C^{2}$ norm.}
Therefore we can process similarly as previously
for the existence of \textcolor{black}{$(x_{i})_{i\in[1,n]}$,} $t_{1}$ and $x_{a}$ \textcolor{black}{$\in C^{1}([0,t_{1}))$ such that 
\begin{equation}
V_{2}(t)=\lvert \sqrt{f_{1}} E\partial_{t}u_{1}(t,\cdot),..., \sqrt{ f_{n}} E\partial_{t}u_{n}(t,\cdot)  \rvert_{0}=-\sqrt{ f_{1}(x_{a}(t))} (E\partial_{t}\mathbf{u})_{1}(t,x_{a}(t)),\text{ }\forall t\in[0,t_{1}),\\
\label{V3m}
\end{equation} 
}
The only thing that remains to be checked is whether we still have the existence of $x_{b}\in C^{1}([0,t_{2}))$ for some $t_{2}$ positive and independent of $m$. 
Existence of a unique $i_{0}$ and $x_{a_{0}}\in[x_{0}-\varepsilon,x_{0}+\varepsilon]$ such that 
\begin{equation}
\lvert\sqrt{f_{i_{0}}(x_{a_{0}})}u_{i_{0}}(0,x_{a_{0}})\rvert=\sup\limits_{i\in[1,n], x\in[0,L]}\lvert\sqrt{f_{i}}u_{i}(0,\cdot)\rvert
\end{equation} 
is granted by the same argument as previously. 
As $\mathbf{u}(0,x)$ is defined exactly as in the linear case, we still have for our choice of $\chi$
\begin{equation}
\partial_{x}g(0,x_{a_{0}})\neq 0.
\end{equation}
This implies the existence of $x_{b}\in C^{1}([0,t_{2}))$ for some $t_{2}$ positive and independent of $m$.\\

If we look now at the computation of $dV_{2}/dt(0)$ and $dV_{1}/dt(0)$, one \textcolor{black}{has}, 
\textcolor{black}{proceeding as in Section \ref{s4} and using \eqref{defm}
\begin{equation}
\begin{split}
\frac{dV_{2}}{dt}(0)&=-\sqrt{f_{1}(x_{1})}\left(E\partial_{tt}^{2}\mathbf{u}+\left(\frac{\partial E}{\partial\mathbf{u}}.\partial_{t}\mathbf{u}\right)\partial_{t}\mathbf{u}\right)_{1}(0,x_{1})\\
 &=+\sqrt{f_{1}(x_{1})}\left(\lambda_{1}(E\partial_{tx}^{2} \mathbf{u})_{1}(0,x_{1})+(E\left(D_{a}+\frac{\partial B}{\partial\mathbf{u}}\right)E^{-1}E\partial_{t}\mathbf{u})_{1}(0,x_{1})\right)\\
 &-\sqrt{f_{1}(x_{1})}\left(\left(\frac{\partial E}{\partial\mathbf{u}}.\partial_{t}\mathbf{u}\right)E^{-1}E\partial_{t}\mathbf{u}\right)_{1}(0,x_{1})\\
&=\sqrt{f_{1}(x_{1})}\left(\lambda_{1}(\partial_{x}(E\partial_{t} \mathbf{u}))_{1}(0,x_{1})+(R(\mathbf{u},x)E\partial_{t}\mathbf{u})_{1}(0,x_{1})\right)\\
&-\sqrt{f_{1}(x_{1})}\left((\lambda\left(\frac{\partial E}{\partial \mathbf{u}}\partial_{x}\mathbf{u}+\partial_{x}E\right)\partial_{t}\mathbf{u})_{1}(0,x_{1})+\left(\left(\frac{\partial E}{\partial\mathbf{u}}.\partial_{t}\mathbf{u}\right)E^{-1}E\partial_{t}\mathbf{u}\right)_{1}(0,x_{1})\right)
 \end{split}
\end{equation} 
where $R=E(D_{a}+\frac{\partial B}{\partial \mathbf{u}})E^{-1}$. Therefore from the definition of $\mathbf{u}(0,x)$ given by (\ref{y11}), one has
\begin{equation}
\begin{split}
\frac{dV_{2}}{dt}(0)
 &=-\sqrt{f_{1}(x_{1})}
\left(-\left(\Lambda_{1}+\frac{l(\mathbf{u}(0,x_{1}),x_{1})}{m}\right)(\partial_{x}(E\partial_{t} \mathbf{u}))_{1}(0,x_{1})+\left(\left(R(\mathbf{0},x)+\frac{v(\mathbf{u}(0,x_{1}),x_{1})}{m}\right)E\partial_{t}\mathbf{u}\right)_{1}(0,x_{1})\right)\\ 
& -\sqrt{f_{1}(x_{1})}\left((\lambda\left(\frac{\partial E}{\partial \mathbf{u}}\partial_{x}\mathbf{u}+\partial_{x}E\right)\partial_{t}\mathbf{u})_{1}(0,x_{1})+\left(\left(\frac{\partial E}{\partial\mathbf{u}}.\partial_{t}\mathbf{u}\right)E^{-1}E\partial_{t}\mathbf{u}\right)_{1}(0,x_{1})\right),\\
 \end{split}
\end{equation} 
} where $l$ and $v$ are bounded functions on $\mathcal{B}_{\eta_{3}}\times[0,L]$ with a bound independent of $m$ from (\ref{uniformaly1}). 
Hence\textcolor{black}{, using this together with \eqref{condId} and noting that $R(\mathbf{0},x)=M(\mathbf{0},x)$,
\begin{equation}
\begin{split}
\frac{d\textcolor{black}{V}_{2}}{dt}(0)
&=-\sqrt{f_{1}(x_{1})}\left((\Lambda_{1}+\frac{l(\mathbf{u}(0,x_{1}),x_{1})}{m})\frac{(\sqrt{f_{1}})'}{\sqrt{f_{1}}}(E\partial_{t}\mathbf{u})_{1}(0,x_{1})\right.\\
&\left.-\sum\limits_{j=1}^{n}(M_{1j}(0,x_{1})+\frac{v_{1j}(\mathbf{u}(0,x_{1}),x_{1})}{m})(E\partial_{t}\mathbf{u})_{j}(0,x_{1})\right)+O\left(|\mathbf{u}|_{1}^{2}\right)\\
 &=-\sqrt{f_{1}(x_{1})}\left(\frac{\Lambda_{1}f_{1}'}{2 f_{1}}(E\partial_{t}\mathbf{u})_{1}(0,x_{1})\left(1+O\left(\frac{1}{m}\right)\right)-\sum\limits_{j=1}^{n}\left(M_{1j}(0,x_{1})+O\left(\frac{1}{m}\right)\right)(E\partial_{t}\mathbf{u})_{j}(0,x_{1})\right)\\
 &+O\left(|\mathbf{u}|_{1}^{2}\right).
 \end{split}
 \label{dV2nl}
\end{equation}
where the $O$ does not depends on $u^{0}_{1}$ but only on an upper bound of $u^{0}_{1}$ (we can choose $\eta_{0}$ for instance). Observe that, from \eqref{y11}, \eqref{withZ}, 
and the fact that $Z=O(1/m)$ for the $C^{2}$ norm, 
}
we can proceed as previously and we obtain (\ref{Cj}) \textcolor{black}{with an additional $O(|\mathbf{u}|_{1}^{2})$}. Similarly as previously we can obtain
\begin{equation}
\sqrt{ f_{i_{0}}(x_{a_{0}})} \partial_{t} u_{i_{0}}(0,x_{a_{0}})=O\left(\frac{1}{m}\right).
\label{dV1nl}
\end{equation}
The rest of the proof 
\textcolor{black}{to get \eqref{expression1}--\eqref{ineq22}} can then be done identically 
as all the relations used in the proof still hold in the nonlinear case.
\textcolor{black}{But actually looking at \eqref{expression1}--\eqref{ineq22}, together with \eqref{quantitenonlinear}, \eqref{unl}, \eqref{dV2nl}--\eqref{dV1nl}, there exists $a>0$ independant of $m$ and $C$ independant of $m$ and $u_{1}^{0}$ such that for any $m>m_{9}$,
\begin{equation}
\frac{dV_{1}}{dt}+\frac{dV_{2}}{dt}\geq a u_{1}^{0}-C(|u_{1}^{0}|^{2}).
\end{equation} 
Thus, there exists $\eta_{2}>0$ independant of $m$ such that, for any $m>m_{9}$,
\begin{equation}
 \frac{dV}{dt}>0
\end{equation} 
which ends the proof in the nonlinear case.}

\subsection{Extension of the proof to the $C^{q}$ norm}
\label{extension}
To be able to extend the proof for the $C^{q}$ norm one should first define the corresponding compatibility conditions of order $q-1$ that are given for instance in \cite{Coron1D} at (4.136) and see also (4.137)-(4.142). Then one only needs to realize that if we now consider the state $\mathbf{y}=(\mathbf{u},\partial_{t}\mathbf{u}, ..., \partial_{t}^{q-1}\mathbf{u})$, 
 $\mathbf{y}$ is still the solution of a quasilinear hyperbolic system 
of the form
\begin{equation}
\mathbf{y}_{t}+A_{1}(\mathbf{y},x)\mathbf{y}_{x}+M_{1}(\mathbf{y},x).\mathbf{y}+C=0
\end{equation}
where $\lvert C_{i}\rvert_{0}=O\left(\lvert u,...,\partial^{i-1}_{t}u \rvert_{0}^{2}\right)$, and where the principal matrix $A_{1}$ verifies
\begin{equation}
A_{1}=\begin{pmatrix} A(\mathbf{u},x) & (0) & ...\\
(0) & A(\mathbf{u},x) & (0) & ...\\
(0)& (0) & A(\mathbf{u},x) & ...\\
... & ... & ... & ...
\end{pmatrix}
\end{equation}
and is therefore block diagonal with blocks that are all $A$ as previously. 
Similarly $M_{1}(0,x)$ is also block diagonal with blocks that are all $M(0,x)$.
Therefore if we consider the following functions
\begin{equation}
 W_{k+1,p}=\left(\int_{0}^{L} \sum\limits_{i=1}^{n}   f_{i}(x)^{p} (\textcolor{black}{E\partial_{t}^{k}\mathbf{u}})_{j}^{2p} e^{-2p\mu s_{i} x} dx \right)^{1/2p},
 \end{equation}
for all $k\in[0, q]$, where $f_{i}$ are chosen as previously, and if we perform as previously \textcolor{black}{(Section \ref{s4} and Appendix \ref{W2derivative})}, we have existence of $C_{k}>0$ constants, $\eta_{k} > 0$ and $p_{k}\in\mathbb{N}^{*}$ such that for all $p > p_{k}$ and $\lvert \mathbf{y}\rvert_{1} < \eta_{k}$ we have relations of the type
\begin{equation}
\frac{dW_{k+1,p}}{dt}\leq-\frac{\mu \alpha_{0}}{2} W_{k+1,p} +C_{k+1}\sum\limits_{1}^{k} W_{r+1,p}\lvert\mathbf{y}\rvert_{1}.
\end{equation}
for all $k\in[0,q]$. Thus denoting $W_{p}=\sum\limits_{k=0}^{q} W_{k+1,p}$, there exists $C>0$ constant,
$p_{l}\in\mathbb{N}^{*}$ and $\eta_{l}>0$ such that for all $p > p_{l}$ and $\lvert\mathbf{y}\rvert_{1} < \eta_{l}$,
\begin{equation}
\frac{dW_{p}}{dt}\leq-\frac{\mu \alpha_{0}}{2} W_{p} +CW_{p}\lvert\mathbf{y}\rvert_{1}.
\end{equation}
and we could perform as previously to obtain the exponential decay of 
$V=\sum\limits_{k=0}^{q} V_{k+1}$ where $V_{k+1}=\lvert \textcolor{black}{\sqrt{f_{1}}}(\partial_{t}^{k}\mathbf{u})_{1},...,\textcolor{black}{\sqrt{f_{n}}}(\partial_{t}^{k}\mathbf{u})_{n}\rvert_{0}$ and therefore stability for the $C^{q}$ norm.

\bibliographystyle{plain}
\bibliography{Biblio2_11}

\begin{thebibliography}{10}

\bibitem{Disturbance}
Ole~Morten Aamo.
\newblock Disturbance rejection in {$2\times 2$} linear hyperbolic systems.
\newblock {\em IEEE Trans. Automat. Control}, 58(5):1095--1106, 2013.

\bibitem{CoronBastin22}
Georges Bastin and Jean-Michel Coron.
\newblock {On boundary feedback stabilization of non-uniform linear 2$\times$ 2
  hyperbolic systems over a bounded interval}.
\newblock {\em Systems \& Control Letters}, 60(11):900--906, 2011.

\bibitem{Coron1D}
Georges Bastin and Jean-Michel Coron.
\newblock {\em Stability and boundary stabilization of 1-{D} hyperbolic
  systems}, volume~88 of {\em Progress in Nonlinear Differential Equations and
  their Applications}.
\newblock Birkh\"auser/Springer, [Cham], 2016.
\newblock Subseries in Control.

\bibitem{CoronBastinPI}
Georges Bastin, Jean-Michel Coron, and Simona~Oana Tamasoiu.
\newblock Stability of linear density-flow hyperbolic systems under {PI}
  boundary control.
\newblock {\em Automatica J. IFAC}, 53:37--42, 2015.

\bibitem{Chanson}
Hubert Chanson.
\newblock {\em Hydraulics of open channel flow}.
\newblock Butterworth-Heinemann, 2004.

\bibitem{CoronNL}
Jean-Michel Coron.
\newblock {\em Control and nonlinearity}, volume 136 of {\em Mathematical
  Surveys and Monographs}.
\newblock American Mathematical Society, Providence, RI, 2007.

\bibitem{CoronC1}
Jean-Michel Coron and Georges Bastin.
\newblock Dissipative boundary conditions for one-dimensional quasi-linear
  hyperbolic systems: {L}yapunov stability for the {$C^1$}-norm.
\newblock {\em SIAM J. Control Optim.}, 53(3):1464--1483, 2015.

\bibitem{CorondAndreaBastin}
Jean-Michel Coron, Georges Bastin, and Brigitte d'Andr\'ea Novel.
\newblock Dissipative boundary conditions for one-dimensional nonlinear
  hyperbolic systems.
\newblock {\em SIAM J. Control Optim.}, 47(3):1460--1498, 2008.

\bibitem{CoronBV}
Jean-Michel Coron, Sylvain Ervedoza, Shyam~Sundar Ghoshal, Olivier Glass, and
  Vincent Perrollaz.
\newblock Dissipative boundary conditions for {$2\times 2$} hyperbolic systems
  of conservation laws for entropy solutions in {BV}.
\newblock {\em J. Differential Equations}, 262(1):1--30, 2017.

\bibitem{CoronNguyen}
Jean-Michel Coron and Hoai-Minh Nguyen.
\newblock Dissipative boundary conditions for nonlinear 1-{D} hyperbolic
  systems: sharp conditions through an approach via time-delay systems.
\newblock {\em SIAM J. Math. Anal.}, 47(3):2220--2240, 2015.

\bibitem{deHalleux}
Jonathan de~Halleux, Christophe Prieur, Jean-Michel Coron, Brigitte d'Andr\'ea
  Novel, and Georges Bastin.
\newblock Boundary feedback control in networks of open channels.
\newblock {\em Automatica J. IFAC}, 39(8):1365--1376, 2003.

\bibitem{DickGugatLeugering}
Markus Dick, Martin Gugat, and Günter Leugering.
\newblock Classical solutions and feedback stabilization for the gas flow in a
  sequence of pipes.
\newblock {\em Networks and Heterogeneous Media}, 5(4):691--709, 2010.

\bibitem{greenberg}
James~M. Greenberg and Tatsien Li.
\newblock The effect of boundary damping for the quasilinear wave equation.
\newblock {\em J. Differential Equations}, 52(1):66--75, 1984.

\bibitem{GasfanGugat}
Martin Gugat, Markus Dick, and G\"unter Leugering.
\newblock Gas flow in fan-shaped networks: classical solutions and feedback
  stabilization.
\newblock {\em SIAM J. Control Optim.}, 49(5):2101--2117, 2011.

\bibitem{LeugeringH2}
Martin Gugat, G\"unter Leugering, and Ke~Wang.
\newblock Neumann boundary feedback stabilization for a nonlinear wave
  equation: {A} strict {$H^2$}-{L}yapunov function.
\newblock {\em Math. Control Relat. Fields}, 7(3):419--448, 2017.

\bibitem{parametre}
Philip Hartman.
\newblock {\em Ordinary differential equations}.
\newblock John Wiley \& Sons, Inc., New York-London-Sydney, 1964.

\bibitem{LyapH2}
Amaury Hayat and Peipei Shang.
\newblock {A quadratic Lyapunov function for the Saint-Venant equations with
  arbitrary friction and space-varying slope in the $H^{2}$ norm}.
\newblock {\em preprint}, 2017.

\bibitem{Robustness}
Hassan~K. Khalil.
\newblock {\em Nonlinear systems}.
\newblock Macmillan Publishing Company, New York, 1992.

\bibitem{LeugeringSchmidt}
Guenter Leugering and E.~J. P.~Georg Schmidt.
\newblock On the modelling and stabilization of flows in networks of open
  canals.
\newblock {\em SIAM J. Control Optim.}, 41(1):164--180, 2002.

\bibitem{ControllabilityLi2}
Tatsien Li.
\newblock {\em Controllability and observability for quasilinear hyperbolic
  systems}, volume~3 of {\em AIMS Series on Applied Mathematics}.
\newblock American Institute of Mathematical Sciences (AIMS), Springfield, MO;
  Higher Education Press, Beijing, 2010.

\bibitem{Tatsien}
Tatsien Li, Bopeng Rao, and Zhiqiang Wang.
\newblock Exact boundary controllability and observability for first order
  quasilinear hyperbolic systems with a kind of nonlocal boundary conditions.
\newblock {\em Discrete Contin. Dyn. Syst.}, 28(1):243--257, 2010.

\bibitem{LiYu}
Tatsien Li and Wen~Ci Yu.
\newblock {\em Boundary value problems for quasilinear hyperbolic systems}.
\newblock Duke University Mathematics Series, V. Duke University, Mathematics
  Department, Durham, NC, 1985.

\bibitem{Qin}
Tie~Hu Qin.
\newblock Global smooth solutions of dissipative boundary value problems for
  first order quasilinear hyperbolic systems.
\newblock {\em Chinese Ann. Math. Ser. B}, 6(3):289--298, 1985.
\newblock A Chinese summary appears in Chinese Ann. Math. Ser. A {{\bf{6}}}
  (1985), no. 4, 514.

\bibitem{slemrod}
Marshall Slemrod.
\newblock Boundary feedback stabilization for a quasilinear wave equation.
\newblock In {\em Control theory for distributed parameter systems and
  applications ({V}orau, 1982)}, volume~54 of {\em Lect. Notes Control Inf.
  Sci.}, pages 221--237. Springer, Berlin, 1983.

\end{thebibliography}

\end{document}